\documentclass[11pt,reqno]{amsart}

\usepackage{amsmath,amssymb,mathrsfs}
\usepackage{graphicx,cite,times}
\usepackage{hyperref}
\hypersetup{
    colorlinks=true,
    linkcolor=blue,
    filecolor=gray,
    urlcolor=blue,
    citecolor=blue,
}

\usepackage[doipre={doi:~}]{uri}
\usepackage{cases}
 \usepackage{dsfont}
 \usepackage{appendix}
 \usepackage{color}
  \usepackage[mathscr]{eucal}
\newtheorem{theo}{Theorem}[section]
\newtheorem{coro}[theo]{Corollary}
\newtheorem{lemm}[theo]{Lemma}
\newtheorem{prop}[theo]{Proposition}
\newtheorem{rema}[theo]{Remark}

\renewcommand{\Re}{\operatorname{Re}}
\renewcommand{\Im}{\operatorname{Im}}
\numberwithin{equation}{section}

\begin{document}

\title[Quasi-periodic response solutions of nonlinear plate models]{
Quasi-periodic response solutions of nonlinear plate models with nonlocal energy
  damping}

\author{Bochao Chen}
\address{School of
Mathematics and Statistics, Center for Mathematics and
Interdisciplinary Sciences, Northeast Normal University, Changchun, Jilin 130024, P.R.China. }
\email{chenbc758@nenu.edu.cn}

\author{Yixian Gao}
\address{School of
Mathematics and Statistics, Center for Mathematics and
Interdisciplinary Sciences, Northeast Normal University, Changchun, Jilin 130024, P.R.China. }
\email{gaoyx643@nenu.edu.cn}

\author{Zhaosheng Feng}
\address{School of Mathematical and Statistical Sciences, University of Texas-RGV, Edinburg, TX 78539, USA
 }
\email{
zhaosheng.feng@utrgv.edu}

\author{Huiying  Liu}
\address{School of Mathematics and Statistics, Center for Mathematics and
Interdisciplinary Sciences, Northeast Normal University, Changchun, Jilin 130024, P.R.China. }
\email{hyliu@nenu.edu.cn}

\thanks{The research of BC was supported in part by  NSFC grants (project number, 12471181)  and the Fundamental Research Funds for the Central Universities (project numbers, 2412022QD032 and 2412023YQ003) The research of YG was supported in part by NSFC grants (project number, 12371187)  and the Science and Technology Development Plan Project of Jilin Province
20240101006JJ. The research of ZF was  partially supported by NSF-DMS 2316952.}



\begin{abstract}
Response solutions are quasi-periodic ones with the same frequency as the forcing term. The present work is devoted to constructing  response solutions for $d$-dimensional nonlinear plate models with nonlocal energy
  damping, which are closely related to damping phenomena in flight structures. 
  For such  models, the main characteristic is that  the dissipation rate depends on the energy strength. By considering a small parameter $\epsilon$  in the domain excluding the origin  and imposing a small quasi-periodic forcing with a  Diophantine frequency vector, we demonstrate the persistence of the corresponding response solution.
We provide  an alternative approach to the contraction mapping principle (cf. \cite{Calleja2017Response,Wang2020Response}) through a combination of reduction  together with  the
Nash--Moser iteration technique. The reason behind this approach lies in the derivative losses caused by the  nonlocal nonlinearity.
\end{abstract}

\maketitle

\section{Introduction}

\subsection*{Problem formulation and a related literature overview}

The present work is devoted to the study of the vibrations of a nonlinear plate model associated with the instantaneous energy
\begin{align*}
\lambda\int_{\mathbb{T}^d}|\Delta u|^2 \mathrm{d}x+\int_{\mathbb{T}^d}|u_t|^2 \mathrm{d}x
\end{align*}
with $\lambda>0$ being a structural positive constant, which  is affected by the displacement $u$ and the velocity $u_t$ (see \cite{Trvares2024Dynamics}). For an integer $d\geq1$,  denote by  $\mathbb{T}^d$  the $d$-dimensional torus. Let $\Delta=\sum^d_{j=1}\partial^2_{x_j}$ represent the Laplace operator. More precisely, we consider the following plate model with nonlocal nonlinear damping and quasi-periodic external forcing:
\begin{align}\label{kirchhoff}
 u_{tt}+\lambda\Delta^2 u-\mu\Delta u+\epsilon u_t+\alpha\Big(\lambda\int_{\mathbb{T}^d}|\Delta u|^2 \mathrm{d}x+\int_{\mathbb{T}^d}|u_t|^2 \mathrm{d}x\Big) u_t=\epsilon^2 g(\omega t,x)
\end{align}
for $(t,x)\in\mathbb{R}\times\mathbb{T}^d$.  Here,  $\mu,\alpha$ are positive constants, $0<\epsilon\ll1$ is a sufficiently small parameter, $\omega=(\omega_1,\cdots,\omega_\nu)\in\mathbb{R}^{\nu}$ is a $\nu$-dimensional forcing frequency vector for some integer $\nu>1$, and $g:\mathbb{T}^\nu\times\mathbb{T}^d\rightarrow\mathbb{R}$ denotes an external forcing term  with zero average, namely,
\begin{align*}%\label{E:average}
\int_{\mathbb{T}^{\nu+d}}g(\varphi,x)\mathrm{d}\varphi\mathrm{d}x=0.
\end{align*}
From a mathematical perspective, Tavares et al. \cite{Trvares2024Dynamics} 
investigated the long-time dynamics including the existence of global attractors and established its quasi-stability.
 Due to the influence of the strength of the energy on the dissipation rate and the wide application of this model, such as in aeroelasticity (see NASA--Air Force reports \cite{balakrishnan1988theory,balakrishnan1989distributed}), this model has attracted significant interest. Other related  models, where the  damping term $u_t$ is replaced by $u_{xxt}$ in nonlinear nonlocal damping, were originally proposed by Balakrishnan and Taylor on flight structures in the one dimensional case; see  \cite{balakrishnan1989distributed}. For further references on extensible beams with nonlocal nonlinear damping of  Balakrishnan--Taylor type, we can refer  to  \cite{jorge2019beam} and the references included in it.

Additionally, a class of new models related to beam equations with nonlocal nonlinear damping has attracted significant recent interest.  For example,
this type of model with a nonlocal nonlinearity of Kirchhoff type was  proposed by Lange and Perla Menzala in \cite{Lange1997Rates}, which are closely connected to nonlinear Schr\"{o}dinger equations with time-dependent dissipation.
They derived the global existence of classical solutions by imposing suitable initial conditions. Since then, many works associated with other partial differential equations (PDEs) with nonlocal nonlinear damping have been presented in literature.  Cavalcanti et al. \cite{cavalcanti2004exponential}  established the global existence and uniform decay rates of solutions to viscoelastic beam equations. Moreover, the authors \cite{Cavalcanti2017Exponential} investigated both the well-posedness and the exponential stability of  solutions for  wave equations. Pucci and Saldi \cite{Pucci2017Asymptotic} investigated the asymptotic stability of solutions for fractional $p$-Laplacian equations. Additionally, there has been growing interest in the exploration of nonlinear beam equations; see \cite{Chen2018Quasi, Chen2019Periodic, Yang2013Extensible, Eliasson2016KAM, ji2023periodic, ge2021KAM}.

Motivated by the aforementioned physical and mathematical studies, the present paper focuses on the existence of  response solutions for beam equations with nonlocal nonlinear damping under Diophantine frequency conditions. A solution of equation \eqref{kirchhoff} is  referred to as a response solution if it is quasi-periodic with the same frequency vector $\omega$ as that of the forcing function $g$ (cf. \cite{Calleja2017Response}), which signifies the simplest harmonic responses to the external forcing.

\subsection*{Other related literature overview}

Numerous studies have addressed the existence of response solutions in various contexts. The research on response solutions was initially focused on ordinary differential equations (ODEs), i.e.,  quasi-periodic forced harmonic oscillator systems with positive damping coefficients. Stoker \cite{Stoker1992nonlinear} first formulated this problem and established the proof of existence, while simultaneously raising the question of whether such solutions persist as damping coefficients approach zero.  Subsequently, Moser \cite{Moser1965combination} answered this question in certain cases for sufficiently small damping coefficients; see also the similar conclusions obtained by Braaksma and Broer \cite{Braaksma1987quasiperiodic} together with Friedman \cite{Friedman1967quasi}. Recently, Si et al. showed that the response solution persists in the cases with sufficiently large damping coefficients. We can also refer to \cite{lou2017quasi,xu2020stoker,ma2023response} for other studies related to the existence of response solutions in the case where damping coefficients vanish. One of the challenges we face when discussing the problem related to the preservation of response solutions  in quasi-periodically forced systems is the so-called small denominator problems. Hence, the proof of the above results relies on Kolmogorov--Arnold--Moser (KAM) theory if the frequency vector is of Diophantine or Liouvillean type. Moreover, Calleja et al. \cite{Calleja2013Construction} and  Gentile and Vaia \cite{gentile2021response} found  alternative approaches to KAM theory using the contraction mapping principle and the properties of continued fractions, respectively.

Before introducing some related work on response solutions to PDEs, we have to refer to the pioneering bifurcation result given by Rabinowitz \cite{Rabinowitz1967periodic,Rabinowitz1968periodic}, for fully nonlinear forced wave equations with small damping. The main challenges encountered in the study of response solutions to PDEs are inconsistent with ODEs. Unlike ODEs, there are two main approaches to deal with response solutions to Hamiltonian PDEs. The first approach involves applying the infinite-dimensional KAM theory; see \cite{cheng2020response,sun2018quasi}. The other one, known as the  Craig--Wayne--Bourgain (CWB) method, is founded on the Lyapunov--Schmidt procedure as well as the Nash--Moser theorem; see \cite{Bourgain1994construction,Baldi2008Forced,craig1993Newton}. Additionally, these studies related
to response solutions to PDEs with the Kirchhoff term could be found in \cite{corsi2018Quasi,Montalto2017quasi}. Recent years have seen growing interest in response solutions for damped PDEs.  Concerning PDEs subject to very strong damping, Calleja et al. \cite{Calleja2017Response} were inspired by  the  procedure  adopted  in \cite{Calleja2013Construction} and established the  existence of response solutions for wave equations. Correspondingly,  Wang and de la Llave \cite{Wang2020Response} presented  such solutions for ill-posed Boussinesq equations indeed exist.

\subsection*{Mathematical challenge}

The primary objective of this work is to investigate the persistence of response solutions to equation \eqref{kirchhoff}. Before addressing the technical challenges, we first examine the limitations of the methods developed  \cite{Calleja2017Response,Wang2020Response}.
 In the analytic or differentiable  case, the authors transformed the original equation into the fixed point equation of the form $v_\epsilon=\mathcal{T}_{\epsilon}(v_\epsilon)$, considered $\epsilon$
in a domain excluding  the origin $\epsilon=0$, and employed the contraction mapping principle since the corresponding operator $\mathcal{T}_{\epsilon}$ was just from a certain function space to itself. However, the techniques and methods  developed  in \cite{Calleja2017Response, Wang2020Response} can not be directly applied to our model \eqref{kirchhoff}. The main challenge in performing these methods is the loss of derivatives caused by the nonlinearity. To overcome this fundamental difficulty,  we employ an alternative approach based on Nash-Moser iteration. The key ingredient is to verify the invertibility of the linearized operator and to provide an appropriate estimate for its inverse operator.
To achieve this goal, we are devoted to constructing  an invertible transformation, which conjugates the linearized operator to an operator $\mathcal{D}$ with constant coefficients plus a remainder $\tilde{\mathcal{R}}$. In this process, we encountered the challenge known as the characteristic ``small divisors problem" arising from resonance phenomena. In order to overcome this difficulty, we have to assume that the frequency vector is Diophantine. On the other hand, the principal part $\mathcal{D}$  admits a diagonal representation with respect to  Fourier bases in  both temporal and spatial variables. Through careful analysis,  we invert $\mathcal{D}$ and give the regularising nature of $\mathcal{D}^{-1}$ for $\epsilon$  in a domain that does not include the origin $\epsilon=0$. As a consequence,  there is a straightforward consequence that $\tilde{\mathcal{R}}$ is relatively bounded with respect to $\mathcal{D}^{-1}$.
In our study, the mathematical analysis presents significant technical difficulties and challenges. 
The theoretical framework developed here can be extended to establish the existence of response solutions to the singular perturbation structure PDEs with nonlocal energy
  damping such as
\begin{align*}
\frac{1}{\epsilon} u_t+\alpha\Big(\lambda\int_{\mathbb{T}^d}|\Delta u|^2 \mathrm{d}x+\int_{\mathbb{T}^d}|u_t|^2 \mathrm{d}x\Big) u_t,
\end{align*}
which exhibits singular behavior as $\epsilon \rightarrow 0.$

\subsection*{Structure of the paper}
The paper is structured as follows.  Section \ref{sec:1} states the main theorem. In Section \ref{sec:2}, we construct  an invertible transformation to reduce the linearized operator to an operator with constant coefficients plus a remainder under non-resonance conditions. Subsequently,  by considering $\epsilon$ in a domain that excludes the origin $\epsilon=0$ and a priori bounds for the high Sobolev norms of the solution, we prove  the invertibility of the linearized operator, which is the crucial ingredient in the Nash--Moser method. Then we provide a proper estimate for its inverse. In Section \ref{sec:3}, we present  the Nash--Moser iteration scheme and complete the proof of the main theorem.

%-----------------------------------------------------------------------------------------------
%-----------------------------------------------------------------------------------------------

\section{Main results}\label{sec:1}
This section presents the main theorem on the existence of response solutions of  equation \eqref{kirchhoff1}.
Equivalently, consider the  rescaling of equation \eqref{kirchhoff}.
Through the rescaling transformation $u\mapsto\epsilon^{p}u$ with $p\in\left(\tfrac{1}{2},1\right)$
and setting  $\varphi=\omega t$, equation \eqref{kirchhoff} can be transformed into: 
\begin{align}\label{kirchhoff1}
(\omega\cdot\nabla_{\varphi})^2u+\lambda\Delta^2u
-\mu\Delta u+\epsilon(\omega\cdot\nabla_{\varphi})u
+\epsilon^{2p}\alpha\big(\lambda\int_{\mathbb{T}^d}|\Delta u|^2 \mathrm{d}x\nonumber\\
+\int_{\mathbb{T}^d}|(\omega\cdot\nabla_{\varphi}) u|^2 \mathrm{d}x\big)(\omega\cdot\nabla_{\varphi})u
=\epsilon^{2-p}g(\varphi,x)
\end{align}
for $(\varphi,x)\in\mathbb{T}^{\nu+d}$. 

We now introduce some necessary notation and non-resonance conditions. 
For any  $\rho>0, s\geq0$, let
\begin{align*}
\mathbb{T}^{\nu}_{\rho}:=\left\{ \varphi\in \mathbb{C}^\nu/(2\pi\mathbb{Z})^\nu:\Re(\varphi_k)\in\mathbb{T},|\Im(\varphi_k)|\leq\rho,k=1,\cdots,\nu\right\}
\end{align*}
and define the Sobolev space $H^{\rho,s}_0$  as follows:
\begin{align*}
H^{\rho,s}_0:=H^{\rho,s}_0(\mathbb{T}^{\nu+d};\mathbb{R}):=\Big\{&u:\mathbb{T}^{\nu+d}\rightarrow\mathbb{R},
u(\varphi,x)
=\sum_{k\in\mathbb{Z}^{\nu},j\in\mathbb{Z}^d}u_{k,j}e^{{\rm i}(k\cdot\varphi+j\cdot x)},\\
& u_{-k,-j}=\overline{u_{k,j}},\|u\|_{\rho,s}<+\infty,
\int_{\mathbb{T}^{\nu+d}}u(\varphi,x)\mathrm{d}\varphi\mathrm{d}x=0\Big\},
\end{align*}
where $\langle k,j\rangle:=\max\{1,|k|,|j|\}$ and
\begin{align*}
\|u\|^{2}_{\rho,s}:=\sum_{k\in\mathbb{Z}^{\nu},j\in\mathbb{Z}^{d}}|u_{k,j}|^2e^{2\rho(|k|+|j|)}\langle k,j\rangle^{2s}.
\end{align*}
 Obviously, $(H^{\rho,s}_0,\|\cdot\|_{\rho,s})$ is a Banach space as well as a Hilbert space.

\begin{rema}
According to the definition of the space $H^{\rho,s}_0$, one has $u_{0,0}=0$ when $k=0$ and $j=0$.
\end{rema}

In the space $H^{\rho,s}_0$, every function 
  admits a bound analytic extension on the complex strip $|\Im(\varphi_k)|<\rho$,
 with  trace function on $|\Im(\varphi_k)|=\rho$, belonging to $H^{\rho,s}_0(\mathbb{T}^{\nu+d};\mathbb{C})$.  If $s>{(\nu+d)}/{2}$, then the space $H^{\rho,s}_0$ is a Banach algebra with respect to multiplication of functions. More precisely,
\begin{align*}
\|vw\|_{\rho,s}\leq C(s)\|v\|_{\rho,s}\|w\|_{\rho,s}, \quad\forall v,w\in{H}^{\rho,s}_0.
\end{align*}
Similarly, we define the space $H^{\rho,s}_{\varphi,0}:=H^{\rho,s}_{\varphi,0}(\mathbb{T}^{\nu};\mathbb{R})$  consisted of functions depending only on $\varphi$ with the zero mean,  equipped with the norm
\begin{align*}
\|u_0\|^{2}_{\varphi,\rho,s}:=\sum_{k\in\mathbb{Z}^{\nu}}|u_{0,k}|^2e^{2\rho|k|}\langle k\rangle^{2s},
\end{align*}
where $\langle k\rangle:=\max\{1,|k|\}$. Clearly, the space $H^{\rho,s}_{\varphi,0}$ is also a Banach algebra for $s>\nu/2$.

From now on,  we fix
\begin{align*}
s_0:=\lfloor \frac{\nu+d}{2}\rfloor+1,
\end{align*}
where $\lfloor\cdot\rfloor$ stands for the integral part of $\cdot$. Define the operator $(\omega\cdot\nabla_{\varphi})^{-1}$ as
\begin{align*}
(\omega\cdot\nabla_{\varphi})^{-1}1&:=0,\\
(\omega\cdot\nabla_{\varphi})^{-1}e^{\mathrm{i}k\cdot \varphi}&:=\frac{e^{\mathrm{i}k\cdot \varphi}}{\mathrm{i}\omega\cdot k},\quad\forall k\in\mathbb{Z}^\nu\backslash\{0\}.
\end{align*}
In other words, $(\omega\cdot\nabla_{\varphi})^{-1}h$ is the primitive of $h$ with zero average in $\varphi$.
In addition, for all $\delta>0$ small enough, we define the set
\begin{align*}
\Lambda_{\delta}:=\Big\{\epsilon\in\mathbb{R}^+:\tfrac{1}{2}\delta\leq\epsilon\leq\delta\Big\}.
\end{align*}

Furthermore, it is essential to impose a constraint on the magnitude of the frequency $\omega$. Without loss of generality, we select
$\omega\in\mathbb{R}^{\nu}$ such that
\begin{align}\label{f:bound}
|\omega|^2=\sum^\nu_{k=1}\omega^2_{k}\leq 1.
\end{align}
A frequency vector $\omega \in \mathbb{R}^\nu$ is called Diophantine if for some
fixed $\gamma>1$ and $\tau > \nu-1$,
\begin{align}\label{non-resonance0}
|\omega\cdot k|^{-1}\leq\gamma |k|^{\tau},\quad \forall k\in\mathbb{Z}^\nu\backslash\{0\}.
\end{align}
 The non-resonance conditions \eqref{non-resonance0} ensure the absence of resonances and permit the existence of polynomially growing small divisors.

Sometimes, we impose non-resonance conditions which are   weaker than the Diophantine conditions, 
including:
\begin{itemize}
\item Brjuno conditions:
\begin{align*}
\sum_{m \geq 0} \frac{1}{2^{m}} \max_{0 < |k| \leq 2^m, k \in \mathbb{Z}^\nu} \ln|\omega \cdot k|^{-1} < +\infty
\end{align*}
\item
Liouvillean conditions (for some  $\gamma >1$ and complex  strip width $\rho>0$):
\begin{align*}
|\omega\cdot k|^{-1}\leq\gamma e^{\frac{\rho}{M}|k|},\quad \forall k\in\mathbb{Z}^\nu\backslash\{0\}.
\end{align*}
\end{itemize}

Additionally, define the nonlinear operator
\begin{align*}
F(u):=(\omega\cdot\nabla_{\varphi})u(\lambda\int_{\mathbb{T}^d}|\Delta u|^2 \mathrm{d}x+\int_{\mathbb{T}^d}|(\omega\cdot\nabla_{\varphi}) u|^2 \mathrm{d}x).
\end{align*}
Equation \eqref{kirchhoff1} can be written equivalently as
\begin{align}\label{k-equivalence2}	
(\omega\cdot\nabla_{\varphi})^2u+\lambda\Delta^2u-\mu\Delta u+\epsilon(\omega\cdot\nabla_{\varphi})u+\epsilon^{2p} \alpha F(u)=\epsilon^{2-p}g.
\end{align}
Furthermore, equation \eqref{k-equivalence2}  can be formulated abstractly as
\begin{align}\label{E:Func}
\mathcal{F}(\epsilon,\omega,u)=0,
\end{align}
where we define
\begin{align*}	\mathcal{F}(\epsilon,\omega,u):=(\omega\cdot\nabla_{\varphi})^2u+\lambda\Delta^2u-\mu\Delta u+\epsilon(\omega\cdot\nabla_{\varphi})u+\epsilon^{2p}\alpha F(u)-\epsilon^{2-p}g.
\end{align*}
Equivalently, we only have to seek  a torus embedding $\varphi\mapsto u(\varphi,\cdot)$ in  $H^{\rho,s}_0$  satisfying equation \eqref{E:Func}.

The following theorem constitutes the main result of this paper, establishing the existence of response solutions for beam equations with nonlocal nonlinear damping under Diophantine frequency conditions.
\begin{theo}\label{main-result}
Fix  $\rho_0>0,s\geq s_0$, and $p\in\left(\frac{1}{2},1\right)$, and that $\lambda, \mu,\alpha$ are positive constant coefficients of order $\mathcal{O}(1)$. Assume  that the external force term $g$ is in $H^{\rho_0,s}_0$ %with zero average
 and the frequency vector $\omega$ satisfies \eqref{f:bound}. Under that the non-resonance conditions  \eqref{non-resonance0} hold for some fixed $\gamma>1$ and $\tau > \nu-1$, then there exists some constant $\kappa\in(0,1)$ small enough such that for any $\delta>0$ with $\delta^{\min\{2p-1,1-p\}}\gamma\leq\kappa$, when $\epsilon$ belongs to $\Lambda_{\delta}$, equation \eqref{E:Func}, which is equivalent to equation \eqref{kirchhoff}, admits a  zero-mean  solution  $u(\epsilon,\omega)\in H^{{\rho_0}/{2},s}_0$, which is quasi-periodic in time.
\end{theo}

In fact, equation \eqref{E:Func} can be solved via the Nash--Moser iteration, as shown in Lemma \ref{le:iterative}, and $\rho_0$ is a constant denoting the width of the complex strip at initial step of iteration.

%----------------------------------------------------------------------------------------------
%-----------------------------------------------------------------------------------------------

\section{Reducibility and invertability of the linearized operator}\label{sec:2}

In this section, we are devoted to verifying the invertibility of the linearized operator and establishing  estimates of its inverse operator.
 Due to the variable coefficients in the linearized operator, we need to construct an invertible transformation which conjugates it to a constant-coefficient operator plus a remainder term.

According to Lemma \ref{le:composition}, the linearized operator $\mathcal{L}(\epsilon,\omega,u)$ associated with \eqref{E:Func} is
\begin{align}
\mathcal{L}(\epsilon,\omega,u)h
&=\mathrm{D}\mathcal{F}(\epsilon,\omega,u)[h]\nonumber\\
&=\frac{\mathrm{d}}{\mathrm{d}\xi}
\mathcal{F}(\epsilon,\omega,u+\xi h)\Big|_{\xi=0}\nonumber\\ &=(\omega\cdot\nabla_\varphi)^2h+\lambda\Delta^2h-\mu\Delta h+\epsilon(\omega\cdot\nabla_{\varphi})h+\epsilon^{2p}\alpha\mathrm{D}F(u)[h]\nonumber\\ &=(\omega\cdot\nabla_\varphi)^2h+\lambda\Delta^2h-\mu\Delta h+\epsilon(\omega\cdot\nabla_{\varphi})h+\epsilon^{2p}\alpha b(\varphi)(\omega\cdot\nabla_{\varphi})h+\epsilon^{2p}\alpha\mathcal{R}h,\label{linearized}
\end{align}
where the coefficient $b$ and the  operator $ \mathcal{R}$  are given by
\begin{align}
b(\varphi)&=\lambda\int_{\mathbb{T}^d}|\Delta u|^2\mathrm{d}x+\int_{\mathbb{T}^d}|(\omega\cdot\nabla_{\varphi}) u|^2 \mathrm{d}x,\nonumber\\
\mathcal{R}h&=\Big(2\lambda\int_{\mathbb{T}^d}\Delta u\cdot{\Delta h}\mathrm{d}x+2\int_{\mathbb{T}^d}(\omega\cdot\nabla_{\varphi})u\cdot(\omega\cdot\nabla_{\varphi})h\mathrm{d}x\Big)(\omega\cdot\nabla_{\varphi})u.\nonumber
\end{align}

\subsection{Reducibility of the linearized operator}

The objective of this subsection focuses on providing the invertible transformation mentioned above.
Our precise statement is as follows:

\begin{prop}\label{prop:change}
Fix  $\rho>0,s\geq s_0$, $p\in\left(\frac{1}{2},1\right)$ and $\tau > \nu-1$.  Suppose that the frequency vector  $\omega$  satisfies the condition of Theorem \ref{main-result}. Provided
\begin{align*}
\|u\|_{\rho,s+\tau+6}\leq1,
\end{align*}
there exists a constant  $\kappa_0\in(0,1)$ small enough such that for any $\delta>0$ with $\delta^{2p}\gamma\leq\kappa_0$, when $\epsilon\in\Lambda_{\delta}$,  we can construct an invertible transformation  $\mathcal{A}$ yielding the conjugation
\begin{align}\label{L2}
\mathcal{A}^{-1}\circ\mathcal{L}(\epsilon,\omega,u)\circ\mathcal{A}
=\tilde{\mathcal{L}}(\epsilon,\omega,u)=\mathcal{D}+\tilde{\mathcal{R}},	
\end{align}
where $\tilde{\mathcal{R}}$ is a zeroth-order operator and
\begin{align}\label{Ope:D} \mathcal{D}=(\omega\cdot\nabla_{\varphi})^2+\lambda\Delta^2-\mu\Delta+\epsilon(\omega\cdot\nabla_{\varphi})+\epsilon^{2p}\iota(\omega\cdot\nabla_{\varphi})%\label{E:D}
\end{align}
with $\iota\equiv\iota(u)$ being a constant related to $u$.
	
Furthermore, there exist two positive constants $K,\tilde{K}$ such that
	
$\mathrm{(i)}$ The constant coefficient $\iota\equiv\iota(u)$ satisfies that
\begin{align}\label{coe}
|\iota|\leq K\|u\|^2_{\rho,s+4}.
\end{align}
	
$\mathrm{(ii)}$ The transformation $\mathcal{A}$ and its inverse $\mathcal{A}^{-1}$ admit the following estimates
\begin{align}
\|\mathcal{A}h\|_{\rho,s}&\leq
K\|h\|_{\rho,s},	\label{c}\\
\|\mathcal{A}^{-1}h\|_{\rho,s}&\leq
K\|h\|_{\rho,s}.\label{d}	
\end{align}
	
$\mathrm{(iii)}$ The remainder $\tilde{\mathcal{R}}$ satisfies that
\begin{align}\label{R2}
\|\tilde{\mathcal{R}}h\|_{\rho,s}\leq\delta^{2p}\gamma
\tilde{K}\|u\|^2_{\rho,s+\tau+6}\|h\|_{\rho,s}.
\end{align}
\end{prop}

\subsection*{Reduction of the first order term}

In order to prove Proposition \ref{prop:change}, we employ  a multiplication transformation 
to convert the variable coefficient $b({\varphi})$ in the operator $\mathcal{L}(\epsilon,\omega,u)$ into a constant. 
Let $\beta:\mathbb{T}^{\nu}_{\rho}\mapsto\mathbb{C} $ be a real analytic function approximating $1$ and  define  the  multiplication operator
\begin{align}
\mathcal{A}: h\mapsto\beta(\varphi)h,\nonumber
\end{align}
where
the unknown function $\beta$ will be determined later. Therefore, its inverse can be  given by
\begin{align}
\mathcal{A}^{-1}:h\mapsto\beta^{-1}(\varphi)h.\nonumber
\end{align}
The following conjugation rules hold:
\begin{align}
\mathcal{A}^{-1}\circ\Delta\circ\mathcal{A}&=\Delta,\nonumber\\	
\mathcal{A}^{-1}\circ\Delta^2\circ\mathcal{A}&=\Delta^2,\nonumber\\
\mathcal{A}^{-1}\circ(\omega\cdot\nabla_{\varphi})\circ\mathcal{A}&=\omega\cdot\nabla_{\varphi}+\beta^{-1}(\varphi)(\omega\cdot\nabla_{\varphi}\beta(\varphi)),\nonumber\\
\mathcal{A}^{-1}\circ(\omega\cdot\nabla_{\varphi})^2\circ\mathcal{A}&=(\omega\cdot\nabla_{\varphi})^2
+2\beta^{-1}(\varphi)(\omega\cdot\nabla_{\varphi}\beta(\varphi))\omega\cdot\nabla_{\varphi}+\beta^{-1}(\varphi)((\omega\cdot\nabla_{\varphi})^2\beta(\varphi)),\nonumber\\
\mathcal{A}^{-1}\circ
b(\varphi)(\omega\cdot\nabla_{\varphi})\circ\mathcal{A}&=\beta^{-1}(\varphi)
b(\varphi)(\omega\cdot\nabla_{\varphi}\beta(\varphi))+b(\varphi)(\omega\cdot\nabla_{\varphi}).\nonumber
\end{align}
Then it is straightforward that
\begin{align}
\mathcal{A}^{-1}\circ\mathcal{L}(\epsilon,\omega,u)\circ\mathcal{A}\nonumber
&=\mathcal{A}^{-1}\circ((\omega\cdot\nabla_{\varphi})^2+\lambda\Delta^2-\mu\Delta+\epsilon(\omega\cdot\nabla_{\varphi})\nonumber\\
&\quad+\epsilon^{2p}\alpha
b(\varphi)(\omega\cdot\nabla_{\varphi})+\epsilon^{2p}\alpha\mathcal{R})\circ\mathcal{A}\nonumber\\ &=\mathcal{A}^{-1}\circ(\omega\cdot\nabla_{\varphi})^2\circ\mathcal{A}+\lambda\mathcal{A}^{-1}\circ\Delta^2\circ\mathcal{A}
-\mu\mathcal{A}^{-1}\circ\Delta\circ\mathcal{A}+\epsilon\mathcal{A}^{-1}\circ(\omega\cdot\nabla_{\varphi})\circ\mathcal{A}\nonumber\\
&\quad+\epsilon^{2p}\alpha\mathcal{A}^{-1}\circ
b(\varphi)(\omega\cdot\nabla_{\varphi})\circ\mathcal{A}+\epsilon^{2p}\alpha\mathcal{A}^{-1}\circ\mathcal{R}\circ\mathcal{A}	\nonumber\\
&=\tilde{\mathcal{L}}(\epsilon,\omega,u)\nonumber,
\end{align}
where
\begin{align}
\tilde{\mathcal{L}}(\epsilon,\omega,u)&:=(\omega\cdot\nabla_{\varphi})^2+\lambda\Delta^2-\mu\Delta+\epsilon(\omega\cdot\nabla_{\varphi})+	\epsilon^{2p}(2\epsilon^{-2p}\beta^{-1}(\varphi)(\omega\cdot\nabla_{\varphi}\beta(\varphi))\nonumber\\
&\quad+\alpha b(\varphi))(\omega\cdot\nabla_{\varphi})+\tilde{\mathcal{R}},	\nonumber\\
\tilde{\mathcal{R}}&:=\epsilon^{2p}\alpha\mathcal{A}^{-1}\circ\mathcal{R}\circ\mathcal{A}+\beta^{-1}(\varphi)((\omega\cdot\nabla_{\varphi})^2\beta(\varphi))
+\epsilon\beta^{-1}(\varphi)(\omega\cdot \nabla_\varphi\beta(\varphi))\nonumber\\
&\quad+\epsilon^{2p}\alpha\beta^{-1}(\varphi)b(\varphi)(\omega\cdot \nabla_\varphi\beta(\varphi))\label{E:R12}.
\end{align}
Suppose that for some constant $\iota$, this holds:
\begin{align}\label{beta-equation}
2\epsilon^{-2p}\beta^{-1}(\varphi)(\omega\cdot\nabla_{\varphi}\beta(\varphi))+\alpha b(\varphi)\equiv\iota.
\end{align}
Observe that
\begin{align}
\beta^{-1}(\varphi)(\omega\cdot\nabla_{\varphi}\beta(\varphi))=\omega\cdot\nabla_{\varphi}\ln(\beta(\varphi)).\nonumber
\end{align}
This leads to
$\int_{\mathbb{T}^{\nu}}(\iota-\alpha b(\varphi))\mathrm{d}\varphi=0$, that is,
\begin{align}
\iota=\frac{\alpha}{(2\pi)^\nu}\int_{\mathbb{T}^{\nu}}b(\varphi)\mathrm{d}\varphi.\nonumber
\end{align}
Thus equation \eqref{beta-equation} can be solved and  explicitly given by
\begin{align}
\beta(\varphi):=\mathrm{exp}\Big(\frac{1}{2}\epsilon^{2p}(\omega\cdot\nabla_{\varphi})^{-1}(\iota-\alpha b(\varphi))\Big).\nonumber
\end{align}
As a result, we rewrite $\tilde{\mathcal{L}}(\epsilon,\omega,u)$ as
\begin{align} \tilde{\mathcal{L}}(\epsilon,\omega,u)=(\omega\cdot\nabla_{\varphi})^2+\lambda\Delta^2-\mu\Delta+\epsilon(\omega\cdot\nabla_{\varphi})+\epsilon^{2p}\iota (\omega\cdot\nabla_{\varphi})+\tilde{\mathcal{R}}.\nonumber
\end{align}

\subsection*{Proof of Proposition \ref{prop:change}}
We are now in a position establish upper-bound estimates for  $|\iota|$ and the operators $\mathcal{A}$, $\mathcal{A}^{-1}$ along with the remainder $\tilde{\mathcal{R}}$ acting on $H^{\rho,s}_0$.
Let us firstly present two crucial estimates arising from composition operator properties.
\begin{lemm}
For $\rho>0$ and $s\geq s_0$,  there exists a constant $K'>0$ depending on $s$ such that
\begin{align}\label{upper:b}
\|b\|_{\varphi,\rho,s}&\leq K'\|u\|^2_{\rho,s+4},\\
\|\mathcal{R}h\|_{\rho,s}&\leq
K'\|u\|^2_{\rho,s+4}\|h\|_{\rho,s}.\label{upper:R}
\end{align}
\end{lemm}
\begin{proof}
Following analogous arguments to those establishing \eqref{a-varphi} and \eqref{B-varphi}, we derive estimates \eqref{upper:b} and \eqref{upper:R}.
\end{proof}

Now let us complete the proof of Proposition \ref{prop:change}.

\begin{proof}
The proof proceeds in three steps.

\textbf{Step 1: Upper-bound estimate of $|\iota|$}.
 It follows from \eqref{2} that for $s\geq s_0$,
\begin{align}
\max_{\varphi\in\mathbb{T}^{\nu}}|b(\varphi)|\leq{C}\|{b}\|_{\varphi,\rho,s},\nonumber
\end{align}
which leads to
\begin{align}
|\iota|\leq
\frac{\alpha}{(2\pi)^\nu}\int_{\mathbb{T}^{\nu}}|
b(\varphi)|\mathrm{d}\varphi\leq
\alpha\max_{\varphi\in\mathbb{T}^{\nu}}|b(\varphi)|\leq{K}\|u\|^2_{\rho,s+4}.\nonumber
\end{align}
This establishes \eqref{coe}.

\textbf{Step 2: Bounds for $\mathcal{A}$, $\mathcal{A}^{-1}$ acting on $H^{\rho,s}_0$}.
Using the Taylor expansion, it is obvious that $e^t=1+t+\frac{t^2}{2!}+\cdots$ for $t$  small enough. Provided the non-resonance conditions \eqref{non-resonance0}, we derive that for $\delta^{2p}\gamma\leq\kappa_0$  small enough,
\begin{align}
\|\beta-1\|_{\varphi,\rho,s}
&\leq \delta^{2p}{C}\left\|(\omega\cdot\nabla_{\varphi})^{-1}(\iota-
\alpha b)\right\|_{\varphi,\rho,s}\exp (\delta^{2p}{C}\|(\omega\cdot\nabla_{\varphi})^{-1}(\iota-
\alpha b)\|_{\varphi,\rho,s})
\nonumber\\
&\leq\delta^{2p}\gamma{C}\|\iota-\alpha b\|_{\varphi,\rho,s+\tau}\exp(\delta^{2p} {C}\gamma\|\iota-
\alpha b\|_{\varphi,\rho,s+\tau}).\nonumber
\end{align}
Because of \eqref{upper:b} and \eqref{coe}, we obtain
\begin{align}
\|\iota-\alpha b\|_{\varphi,\rho,s}\leq C_1\|u\|^2_{\rho,s+4},\nonumber
\end{align}
which leads to
\begin{align}\label{b}
\|\beta-1\|_{\varphi,\rho,s}\leq\delta^{2p}\gamma C_2\|u\|^2_{\rho,s+\tau+4}.
\end{align}
Clearly, it follows from \eqref{b} that \eqref{c} and \eqref{d} hold, respectively.

\textbf{Step 3: Upper-bound estimate of the remainder $\tilde{\mathcal{R}}$ acting on $H^{\rho,s}_0$}.
 If $\|u\|_{\rho,s+\tau+6}\leq1$, then applying \eqref{c}, \eqref{d}, \eqref{upper:b}--\eqref{b}, and Neumann series yields that
\begin{align}	
\|\epsilon^{2p}\alpha\mathcal{A}^{-1}\circ\mathcal{R}\circ\mathcal{A}h\|_{\rho,s}\nonumber
&\leq\delta^{2p}\alpha K\|\mathcal{R}\circ\mathcal{A}h\|_{\rho,s}\nonumber\\
&\leq\delta^{2p}K^2K''\|u\|^2_{\rho,s+4}\|h\|_{\rho,s}.\nonumber
\end{align}
Furthermore, since
\begin{align}		\omega\cdot\nabla_{\varphi}\beta=\omega\cdot\nabla_{\varphi}(\beta-1),\quad(\omega\cdot\nabla_{\varphi})^2\beta=(\omega\cdot\nabla_{\varphi})^2(\beta-1),\nonumber
\end{align}
we have
\begin{align}
\|\beta^{-1}((\omega\cdot\nabla_{\varphi})^2\beta)h\|_{\rho,s}&\leq
C_3\|(\omega\cdot\nabla_{\varphi})^2\beta\|_{\varphi,\rho,s}\|h\|_{\rho,s}\nonumber\\
&\leq
C_3\|\beta-1\|_{\varphi,\rho,s+2}\|h\|_{\rho,s}\nonumber\\
&\leq\delta^{2p}\gamma
K\|u\|^2_{\rho,s+\tau+6}\|h\|_{\rho,s}\nonumber,\\
\|\epsilon\beta^{-1}(\omega\cdot \nabla_\varphi\beta)h\|_{\rho,s}	&\leq\delta\|\beta^{-1}\|_{\varphi,\rho,s}\|(\omega\cdot\nabla_\varphi)(\beta-1)\|_{\varphi,\rho,s}\|h\|_{\rho,s}\nonumber\\
&\leq\delta^{2p+1}\gamma
K\|u\|^2_{\rho,s+\tau+5}\|h\|_{\rho,s}\nonumber,
\end{align}
and
\begin{align}
\|\epsilon^{2p}\alpha\beta^{-1}b(\omega\cdot \nabla_\varphi\beta)h\|_{\rho,s}&\leq\delta^{2p}\alpha\|\beta^{-1}\|_{\varphi,\rho,s}\|b\|_{\varphi,\rho,s}\|\omega\cdot \nabla_\varphi\beta\|_{\varphi,\rho,s}\|h\|_{\rho,s}\nonumber\\
&\leq\delta^{4p}\gamma
K\|u\|^2_{\rho,s+\tau+5}\|h\|_{\rho,s}.\nonumber
\end{align}
Following the above calculations and using \eqref{E:R12}, we can obtain \eqref{R2}, which completes the proof of the lemma.
\end{proof}

\subsection{Invertibility of the linearized operator}

In this subsection, we aim at checking the invertibility of the linearized  operator $\mathcal{L}(\epsilon,\omega,u)$  given by \eqref{linearized} and  giving a proper estimate for its inverse.

According to  equality \eqref{L2}, it can be clearly deduced that
\begin{align}\label{E:identity}
\mathcal{L}(\epsilon,\omega,u)=\mathcal{A}\circ\tilde{\mathcal{L}}(\epsilon,\omega,u)\circ\mathcal{A}^{-1}=\mathcal{A}\circ(\mathcal{D}+\tilde{\mathcal{R}})\circ\mathcal{A}^{-1}.
\end{align}
In Fourier bases in time and space, the main part $\mathcal{D}$ is  diagonal. Then for $\epsilon$  in a domain that does not include the origin $\epsilon=0$, we invert $\mathcal{D}$ and give the estimate for $\mathcal{D}^{-1}$ acting on $H^{\rho,s}_0$.

The following lemma  corresponds to the invertibility of the main part $\mathcal{D}$.

\begin{lemm}\label{le:Dinverse}
Let $\rho>0,s>0$, $p\in\left(\frac{1}{2},1\right)$, and $\tau>\nu-1$ be fixed. If $\|u\|_{\rho,s+\tau+6}\leq1$, when $\epsilon\in\Lambda_{\delta}$ for any $0<\delta\ll1$, then  the operator $\mathcal{D}:H^{\rho,s}_0\rightarrow H^{\rho,s}_0$ (recall \eqref{Ope:D}) is invertible satisfying
\begin{align}
\mathcal{D}^{-1}:H^{\rho,s}_0\rightarrow H^{\rho,s}_0\nonumber
\end{align}
and there exists some constant ${K}_0>0$ such that for $h\in H^{\rho,s}_0$,
\begin{align}\label{E:inverse2}
\|\mathcal{D}^{-1}h\|_{\rho, s}\leq K_0\delta^{-1}\|h\|_{\rho,s}.
\end{align}
\end{lemm}
\begin{proof}
It is straightforward that for $h\in H^{\rho,s}_0$,
\begin{align}
\mathcal{D}h=\sum_{k\in\mathbb{Z}^{\nu}, j\in\mathbb{Z}^d}\Theta_\epsilon(\omega,j)h_{k, j}e^{\mathrm{i}(k\cdot\varphi+ j\cdot x)},\nonumber
\end{align}
where the symbol $\Theta_\epsilon(\omega,j)$ takes the form
\begin{align}
\Theta_\epsilon(\omega,j)=-(\omega\cdot k)^2+\lambda|j|^4+\mu|j|^2+\mathrm{i}(\epsilon^{2p}\iota(\omega\cdot k)+\epsilon(\omega\cdot k))\nonumber
\end{align}
with $\iota\equiv\iota(u)$. It remains to estimate the lower bounds of
\begin{align}
|\Theta_\epsilon(\varsigma,j)|^2=(\lambda|j|^4+\mu|j|^2-\varsigma^2)^2+(\epsilon^{2p}\iota\varsigma+\epsilon\varsigma)^2,\nonumber
\end{align}
where we set $\varsigma:=\omega\cdot k$.

In the subsequent analysis, a positive integer $J_0>1$ will be fixed. Additionally, it is important to observe that $j\neq0$ according to the definition of the space $H^{\rho,s}_0$. We are now in a position to discuss the following two cases.

$\mathbf{Case~a}$: $1\leq \lambda|j|^4+\mu|j|^2\leq J_0$. Observe that
\begin{align*}
\varsigma^2=\lambda|j|^4+\mu|j|^2\in[1,J_0].
\end{align*}
Define two regions in $\varsigma\in\mathbb{R}$ as follows:
\begin{align*}
I^{-}_1&:=[-(1+10^{-3}){J}^{\frac12}_0,-(1-10^{-3})],\nonumber\\
I^{+}_1&:=[1-10^{-3},(1+10^{-3})J^{\frac12}_0].
\end{align*}
For $\varsigma\in \mathbb{R}\setminus (I^{-}_1\cup I^{+}_1)$, one has
\begin{align}
\text{either}\quad 0<|\varsigma|< 1-10^{-3}\quad\text{or}\quad |\varsigma|> (1+10^{-3})J^{\frac12}_0.\nonumber
\end{align}
If $0<|\varsigma|< 1-10^{-3}$, then
\begin{align}
|\Theta_\epsilon(\varsigma,j)|^2\geq(\lambda|j|^4+\mu|j|^2-\varsigma^2)^2\geq10^{-6}. \nonumber
\end{align}
In the latter, one obtains that for $|\varsigma|> (1+10^{-3})J^{\frac12}_0$,
\begin{align}
|\Theta_\epsilon(\varsigma,j)|^2\geq(\lambda|j|^4+\mu|j|^2-\varsigma^2)^2\geq(((1+10^{-3})J_0^\frac12)^2-J_0)^2\geq
10^{-6}J_0^2.\nonumber
\end{align}
Moreover, we can derive that for $\varsigma\in I^{-}_1\cup I^{+}_1$,
\begin{align}
1-10^{-3}\leq|\varsigma|\leq(1+10^{-3})J^{\frac12}_0.\nonumber
\end{align}
Since $\|u\|_{\rho,s+\tau+6}\leq1$, using \eqref{coe} yields that $|\iota|\leq K$. Then we get that for $\delta^{2p-1}K\leq1/2$,
\begin{align} |\Theta_\epsilon(\varsigma,j)|^2\geq(\epsilon^{2p}\iota\varsigma+\epsilon\varsigma)^2\geq\epsilon^2\varsigma^2(1-\epsilon^{2p-1}|\iota|)^2
\geq\frac{1}{16}\delta^2(1-10^{-3})^2.\nonumber
\end{align}

$\mathbf{Case~b}$:  $\lambda|j|^4+\mu|j|^2> J_0$. One has that for $\left|\lambda|j|^4+\mu|j|^2-\varsigma^2\right|\geq\frac12J_0$,
\begin{align}
|\Theta_\epsilon(\varsigma,j)|^2\geq(\lambda|j|^4+\mu|j|^2-\varsigma^2)^2\geq\frac{1}{4}J^2_0.\nonumber
\end{align}
On the other hand, it follows from the inequality $\left|\lambda|j|^4+\mu|j|^2-\varsigma^2\right|<\frac12 J_0$ that
\begin{align}
\frac12 J_0<\lambda|j|^4+\mu|j|^2-\frac12 J_0\leq\varsigma^2\leq\lambda|j|^4+\mu|j|^2+\frac12 J_0.\nonumber
\end{align}
Therefore, we have that for $\delta^{2p-1}K\leq1/2$,
\begin{align} |\Theta_\epsilon(\varsigma,j)|^2\geq(\epsilon^{2p}\iota\varsigma+\epsilon\varsigma)^2\geq\frac{1}{16}\delta^2\varsigma^2\geq\frac{1}{32}\delta^2 J_0.\nonumber
\end{align}
As a result, the operator $\mathcal{D}:H^{\rho,s}_0\rightarrow H^{\rho,s}_0$ is invertible.

Furthermore, we can deduce
\begin{align}
\sum_{k\in\mathbb{Z}^{\nu}, j\in\mathbb{Z}^d}{|h_{k,j}|^2}{|\Theta_\epsilon(\omega,j)|^{-2}} e^{2\rho(|k|+|j|)}\langle k,j\rangle^{2s}\leq K_0^2\delta^{-2}\|h\|^2_{\rho,s}<\infty,\nonumber
\end{align}
where $K_0$ may be taken as $10^6$ and the Fourier coefficients $h_{0,0}=0$.

The proof is complete.
\end{proof}

For convenience,  we write
\begin{align}\label{Ope:L}
\mathcal{L}_{\epsilon,\omega}:=(\omega\cdot\nabla_\varphi)^2+\lambda\Delta^2-\mu\Delta+\epsilon(\omega\cdot\nabla_{\varphi}).
\end{align}
The similar procedure as in the proof of Lemma \ref{le:Dinverse} also provides the invertibility of $\mathcal{L}_{\epsilon,\omega}$.
\begin{lemm}
For $\rho>0$ and $s>0$, if $u$ and $\epsilon$ satisfy the conditions of Lemma \ref{le:Dinverse}, then the operator $\mathcal{L}_{\epsilon,\omega}: H^{\rho,s}_0\rightarrow H^{\rho,s}_0$  is invertible with
\begin{align}\label{E:inverse3}
\|\mathcal{L}^{-1}_{\epsilon,\omega}h\|_{\rho, s}\leq K_0\delta^{-1}\|h\|_{\rho,s},
\end{align}
where $K_0$ can be founded in \eqref{E:inverse2}.
\end{lemm}

Our next goal is to present the invertibility of $\tilde{\mathcal{L}}(\epsilon,\omega,u)$ via Neumann series arguments.

\begin{lemm}\label{le:L2}
Let $\rho>0, s\geq s_0$, and $p\in\left(\frac{1}{2},1\right)$. Suppose that $\|u\|_{\rho,s+\tau+6}\leq1$ for fixed $\tau>\nu-1$. Then there is $\kappa_1\in(0,\kappa_0)$ small enough satisfying that for any $\delta>0$ with $\delta^{2p-1}\gamma\leq\kappa_1$, when $\epsilon$ is in $\Lambda_{\delta}$,  the operator $\tilde{\mathcal{L}}(\epsilon,\omega,u)$ is invertible from  $H^{\rho,s}_0$ to $H^{\rho,s}_0$  with
\begin{align}
\|\tilde{\mathcal{L}}^{-1}(\epsilon,\omega,u)h\|_{\rho,s}\leq 2K_0\delta^{-1}\|h\|_{\rho,s},\nonumber
\end{align}
where $K_0$ is defined in \eqref{E:inverse2}.
\end{lemm}

\begin{proof}
According to Lemma \ref{le:Dinverse}, we decompose
\begin{align}		\tilde{\mathcal{L}}(\epsilon,\omega,u)h=\mathcal{D}h+\tilde{\mathcal{R}}h=\mathcal{D}(\mathrm{Id}+\mathcal{D}^{-1}\tilde{\mathcal{R}})h.\nonumber
\end{align}
For sufficiently small  $\delta^{2p-1}\gamma\leq\kappa_1$, it follows from \eqref{R2} and \eqref{E:inverse2} that
\begin{align*}
\|\mathcal{D}^{-1}\tilde{\mathcal{R}}h\|_{\rho,s}&\leq\delta^{-1} K_0\|\tilde{\mathcal{R}}h\|_{\rho,s}\\
&\leq\delta^{2p-1}\gamma K_0\tilde{K}\|u\|^2_{\rho,s+\tau+6}\|h\|_{\rho,s}\\
&\leq\frac{1}{2}\|h\|_{\rho,s}.
\end{align*}
Then the Neumann series guarantees the invertibility of $\mathrm{Id}+\mathcal{D}^{-1}\tilde{\mathcal{R}}$ with
\begin{align}
\|(\mathrm{Id}+\mathcal{D}^{-1}\tilde{\mathcal{R}})^{-1}h\|_{\rho,s}\leq
2\|h\|_{\rho,s}.\nonumber
\end{align}
As a result,
\begin{align*}		\|\tilde{\mathcal{L}}^{-1}(\epsilon,\omega,u)h\|_{\rho,s}&=\|(\mathrm{Id}+\mathcal{D}^{-1}\tilde{\mathcal{R}})^{-1}\mathcal{D}^{-1}h\|_{\rho,s}\nonumber\\
&\leq2\|\mathcal{D}^{-1}h\|_{\rho,s}\nonumber\\
&\leq2K_0\delta^{-1}\|h\|_{\rho,s},\nonumber
\end{align*}
which completes the proof of the lemma.
\end{proof}

As a consequence, as stated in the following proposition, we give  the invertibility of the linearized  operator $\mathcal{L}(\epsilon,\omega,u)$.

\begin{prop}\label{le:L}
Under the hypotheses of Lemma  \ref{le:L2}, the operator  $\mathcal{L}(\epsilon,\omega,u): H^{\rho,s}_0\rightarrow H^{\rho,s}_0$ is invertible with
\begin{align}
\mathcal{L}^{-1}(\epsilon,\omega,u): H^{\rho,s}_0\rightarrow H^{\rho,s}_0\nonumber
\end{align}
and there exists some constant ${K}_1>0$ such that for $h\in H^{\rho,s}_0$,
\begin{align}\label{L:K1}
\|\mathcal{L}^{-1}(\epsilon,\omega,u)h\|_{\rho,s}\leq{K}_1\delta^{-1}\|h\|_{\rho,s}.
\end{align}
\end{prop}

\begin{proof}
Applying Proposition \ref{prop:change} and Lemma \ref{le:L2}, the operators $\mathcal{A}$, $\mathcal{A}^{-1}$, and $\tilde{\mathcal{L}}(\epsilon,\omega,u)$ are all invertible. 
From \eqref{E:identity} we obtain the factorization
\begin{align}
\mathcal{L}^{-1}(\epsilon,\omega,u)=\mathcal{A}\circ\tilde{\mathcal{L}}^{-1}(\epsilon,\omega,u)\circ\mathcal{A}^{-1}.\nonumber
\end{align}
Therefore, we can derive the estimate 
\begin{align*}
\|\mathcal{L}^{-1}(\epsilon,\omega,u)h\|_{\rho,s}
&\leq K\|\tilde{\mathcal{L}}^{-1}(\epsilon,\omega,u)\circ\mathcal{A}^{-1}h\|_{\rho,s}\\
&\leq2KK_0\delta^{-1}\|\mathcal{A}^{-1}h\|_{\rho,s}\\
&\leq{K}_1\delta^{-1}\|h\|_{\rho,s}.
\end{align*}
The proof is now complete.
\end{proof}

%
%
%%----------------------------------------------------------------------------------------------
%%----------------------------------------------------------------------------------------------
%
\section{Proof of the main theorem}\label{sec:3}

In this section we first construct the Nash--Moser iteration scheme,  and then demonstrate the existence of the response solution to the equivalent equation \eqref{E:Func} of equation \eqref{kirchhoff}. Finally,  we give the proof of the main theorem \ref{main-result}.

\subsection{Solvability of equation (2.5)}

The goal of this subsection is devoted to constructing a Nash--Moser iteration scheme to solve equation \eqref{E:Func}.

Denote by $\mathbb{N}$ the set composed of non-negative integers. Let $n\in\mathbb{N}$ be a non-negative integer.  Set
\begin{align}\label{Nn}
N_n:=N_0 2^n,
\end{align}
where $N_0$ will be fixed in Lemma \ref{le:iterative}. Consider the increasing sequence of finite-dimensional subspaces as follows:
\begin{align}
\textstyle H_{n}:=\left\{u(\varphi,x)=\sum_{\langle k,j\rangle\leq N_n}u_{k,j}e^{\mathrm{i}(k\cdot \varphi+j\cdot x)}\in H^{\rho,0}_0\right\}\subset H^{\rho,s}_0.\nonumber
\end{align}
Correspondingly, we set
\begin{align}
\textstyle H^{\bot}_{n}:=\left\{u(\varphi,x)=\sum_{\langle k,j\rangle> N_n}u_{k,j}e^{\mathrm{i}(k\cdot \varphi+j\cdot x)}\in H^{\rho,0}_0\right\}\subset H^{\rho,s}_0.\nonumber
\end{align}
Denote by $\mathcal{P}_{n}$ and $\mathcal{P}^{\bot}_{n}$ the projectors on  $H_{n}$ and $ H^{\bot}_{n}$, respectively. For all $\rho>0,\tilde{\rho}\geq0,s>0,\tilde{s}\geq0$, the operator $\mathcal{P}_{n}$ possesses the smoothing properties
\begin{align}
\|\mathcal{P}_{n}u\|_{\rho+\tilde{\rho},s+\tilde{s}}\leq e^{2\tilde{\rho}N_n}N^{\tilde{s}}_n\|u\|_{\rho,s},\quad\forall u\in H^{\rho+\tilde{\rho},s+\tilde{s}}_0.\nonumber
\end{align}
The operator $\mathcal{P}^{\perp}_{n}$ satisfy that for $0<\tilde{\rho}\leq\rho$ and $s>0$,
\begin{align}%\label{smooth2}
\|\mathcal{P}^{\perp}_{n}u\|_{\tilde{\rho},s}\leq e^{-(\rho-\tilde{\rho})N_n}\|u\|_{\rho,s},\quad\forall u\in H^{\rho,s}_0.\nonumber
\end{align}
Moreover, denote by $\rho_0$  the width of the complex strip at initialization. We define the sequence by
\begin{align}\label{sequence}
\rho_{n+1}:=\rho_n-\frac{\theta}{(n+1)^2},
\end{align}
where $\theta=3\rho_0/\pi^2$; see also \cite{Baldi2008Forced}. It is clear that
\begin{align*}%\label{E:lambda}
\rho_0>\rho_1>\rho_2>\rho_3>\cdots>{\rho_0}/{2}.
\end{align*}

We are in a position to introduce the inductive lemma.

\begin{lemm}[Inductive lemma]\label{le:iterative}
Let $\rho_0>0,s\geq s_0$, and $p\in\left(\frac{1}{2},1\right)$.  %and $\Theta$ be a positive constant depending on $s$ and $g$.
 For all $n\in\mathbb{N}$, for some $\kappa\in(0,1)$ small enough satisfying that for any $\delta>0$ with $\delta^{2p-1}\gamma\leq\kappa$, when $\epsilon\in\Lambda_{\delta}$, then we can construct a sequence $u_{n}(\epsilon,\omega)\in H_{n}$, which is a solution of equation
\begin{align}
\mathcal{P}_n\mathcal{F}(\epsilon,\omega,u)=0,\nonumber
\end{align}
namely,
\begin{align}\label{t-equation}
\mathcal{L}_{\epsilon,\omega}u+\epsilon^{2p}\alpha\mathcal{P}_{n}F(u)-\epsilon^{2-p}\mathcal{P}_ng=0, \tag{$\mathcal{P}_{n}$}
\end{align}
and satisfies that for some fixed $\tau>\nu-1$,
\begin{align}\label{prior}
\|u_{n}\|_{\rho_{n+1},s+\tau+6}\leq1,
\end{align}
and
\begin{align}
\textstyle\|u_{0}\|_{\rho_0,s}&\leq \delta^{{\min\{2p-1,1-p\}}}\bar\Theta N^{12}_0,\label{initial}\\
\|u_l-u_{l-1}\|_{\rho_l,s}&\leq \delta^{{\min\{2p-1,1-p\}}}\bar\Theta N^{12}_{0}e^{-\chi^{l}}, \quad\forall1\leq l\leq n,\label{iteration}
\end{align}
where $ \chi$ is taken as $3/2$ and $\bar \Theta$ is a positive constant depending on $s$ and $g$, provided by
\begin{align}
\bar\Theta=(K_0+2K_1)(\tilde{C}(s)+\|g\|_{\rho_0,s}),\nonumber
\end{align}
where $K_0$ and $K_1$ are given by \eqref{E:inverse2} and \eqref{L:K1}, respectively.
\end{lemm}

\begin{proof}
The proof proceeds by induction with two main steps:
	
$\textbf{Step 1: Initialization.}$  From formula \eqref{E:inverse3}, solving equation $(\mathcal{P}_{0})$ is transformed into the fixed point problem of $u=\mathcal{U}_0(u)$, where	
\begin{align*}
\textstyle \mathcal{U}_0:H_{0}&\rightarrow H_{0},\\
u&\mapsto \mathcal{L}^{-1}_{\epsilon,\omega}(-\epsilon^{2p}\alpha\mathcal{P}_0F(u)+\epsilon^{2-p}\mathcal{P}_0g),
\end{align*}
where $\mathcal{L}_{\epsilon,\omega}$ is defined by \eqref{Ope:L}.

In the following lemma, we check that $\mathcal{U}_0$ is a contraction mapping.

\begin{lemm}
Let $\rho_0>0,s\geq s_0$, and $p\in\left(\frac{1}{2},1\right)$. For fixed $\gamma>\nu-1$, there exists $\kappa_2\in(0,\kappa_1)$ sufficiently small such that for any $\delta>0$ with $\delta^{\min\{2p-1,1-p\}}\gamma\leq\kappa_2$, when $\epsilon$ is in $\Lambda_{\delta}$, the mapping $\mathcal{U}_0$ is a contraction in the set $\mathcal{B}(0,\eta_0)$ given by
\begin{align}
\mathcal{B}(0,\eta_0):=\left\{u\in H_{0}:\|u\|_{\rho_0,s}\leq \eta_0\right\}\nonumber%\quad\text{with}\quad
\end{align}
with $\eta_0:=\delta^{\min\{2p-1,1-p\}}\bar\Theta N^{12}_0$.
\end{lemm}
\begin{proof}
It follows from \eqref{E:nonlinearity1} and \eqref{E:inverse3} that for $\delta^{\min\{2p-1,1-p\}}\gamma\leq\kappa_2$ small enough,
\begin{align*}
\|\mathcal{U}_0(u)\|_{\rho_0,s}
&\leq \delta^{-1} K_0\|-\epsilon^{2p}\alpha\mathcal{P}_0F(u)+\epsilon^{2-p}\mathcal{P}_0g\|_{\rho_0,s}\\
&\leq\delta^{-1}K_0(\delta^{2p} \tilde{C}(s)\|u\|^3_{\rho_0,s+4}+\delta^{2-p}\|\mathcal{P}_0g\|_{\rho_0,s})\\
&\leq\delta^{\min\{2p-1,1-p\}} K_0N^{12}_0(\tilde{C}(s)\|u\|^3_{\rho_0,s}+\|g\|_{\rho_0,s})\\
&\leq\delta^{\min\{2p-1,1-p\}}\bar\Theta N^{12}_0.
\end{align*}
Moreover, observe that
\begin{align*}
\mathrm{D}\mathcal{U}_0(u)[h]&=\frac{\mathrm{d}}{\mathrm{d}\xi}
\mathcal{U}_0(u+\xi h)\Big|_{\xi=0}\nonumber\\
&=\frac{\mathrm{d}}{\mathrm{d}\xi}\mathcal{L}^{-1}_{\epsilon,\omega}(-\epsilon^{2p}\alpha\mathcal{P}_0F(u+\xi h)+\epsilon^{2-p}\mathcal{P}_0g)\Big|_{\xi=0}\nonumber\\
&=\lim\limits_{\xi \rightarrow 0}\frac{1}{\xi} \left[\mathcal{L}^{-1}_{\epsilon,\omega}(-\epsilon^{2p}\alpha\mathcal{P}_0F(u+\xi h)+\epsilon^{2-p}\mathcal{P}_0g) -\mathcal{L}^{-1}_{\epsilon,\omega}(-\epsilon^{2p}\alpha\mathcal{P}_0F(u)+\epsilon^{2-p}\mathcal{P}_0g) \right]\nonumber\\
&=-\epsilon^{2p}\alpha\mathcal{L}^{-1}_{\epsilon,\omega}\mathcal{P}_0(\mathrm{D}F(u)[h]).\nonumber
\end{align*}
Combining this with \eqref{E:nonlinearity2} gives that
\begin{align*}
\|\mathrm{D}\mathcal{U}_0(u)[{h}]\|_{\rho_0,s}
&\leq\delta^{2p-1} \alpha K_0\|\mathcal{P}_0(\mathrm{D}F(u)[{h}])\|_{\rho_0,s}\nonumber\\
&\leq\delta^{2p-1}\bar\Theta N^{12}_0\|u\|^2_{\rho_0,s}\|{h}\|_{\rho_0,s}\nonumber\\
&\leq\frac12\|{h}\|_{\rho_0,s}.\nonumber
\end{align*}
Therefore,  the mapping $\mathcal{U}_0$ is a contraction in $\mathcal{B}(0,\eta_0)$, thereby completing the proof.
\end{proof}
	
For the convenience of describing the problem, denote by $u_0$ the unique solution of equation $(\mathcal{P}_{0})$ in $\mathcal{B}(0,\eta_0)$. Additionally, there exists $\kappa_3\in(0,\kappa_2)$  sufficiently small  such that if $\delta^{\min\{2p-1,1-p\}}\gamma\leq\kappa_3$, then
\begin{align}
\|u_0\|_{\rho_1,s+\tau+6}\stackrel{\eqref{1}}{\leq} \frac{1}{\theta^{\tau+6}}\big(\frac{\tau+6}{e}\big)^{\tau+6} \|u_0\|_{\rho_0,s}\stackrel{\eqref{initial}}{\leq}\delta^{\min\{2p-1,1-p\}}\frac{1}{\theta^{\tau+6}}\big(\frac{\tau+6}{e}\big)^{\tau+6}\bar\Theta N^{12}_0\leq1,\nonumber
\end{align}
which implies that \eqref{prior} holds for $n=0$.
	
Thus, we complete the proof of initialization.
	
$\textbf{Step 2: Iteration.}$ Suppose that we could obtain a solution $u_{n}\in H_n$ of equation \eqref{t-equation} satisfying formulas \eqref{prior}--\eqref{iteration}.
	
Our final objective is to look for a solution $u_{n+1}\in H_{n+1}$ of equation
\begin{align}\label{Pn+1}
\mathcal{L}_{\epsilon,\omega}u+\epsilon^{2p}\alpha\mathcal{P}_{n+1}F(u)-\epsilon^{2-p}\mathcal{P}_{n+1}g=0 \tag{$\mathcal{P}_{n+1}$}
\end{align}
satisfying \eqref{prior} and \eqref{iteration} at the ($n+1$)-th step.

Denote by
\begin{align}
u_{n+1}=u_{n}+h_{n+1},\nonumber
\end{align}
\text{with} $h_{n+1}\in H_{{n+1}}$, a solution of equation \eqref{Pn+1}. It follows from equation \eqref{t-equation} that
\begin{align*}		&\mathcal{L}_{\epsilon,\omega}(u_{n}+h_{n+1})+\epsilon^{2p}\alpha\mathcal{P}_{n+1}F(u_{n}+h_{n+1})-\epsilon^{2-p}\mathcal{P}_{n+1}g\\	&=\mathcal{L}_{\epsilon,\omega}u_{n}+ \mathcal{L}_{\epsilon,\omega} h_{n+1}+\epsilon^{2p}\alpha\mathcal{P}_{n+1}F(u_{n}+h_{n+1})-\epsilon^{2-p}\mathcal{P}_{n+1}g\\
&=\mathcal{L}_{n+1}(\epsilon,\omega,u_n)h_{n+1}+R_{n}(h_{n+1})+r_{n},
\end{align*}
where
\begin{align}%\label{truncate}	
\mathcal{L}_{n+1}(\epsilon,\omega,u_n)h_{n+1}&:=\mathcal{L}_{\epsilon,\omega}h_{n+1} +\epsilon^{2p}\alpha \mathcal{P}_{n+1}(\mathrm{D}F(u_n)[h_{n+1}]),\nonumber\\	R_{n}(h_{n+1})&:=\epsilon^{2p}\alpha\mathcal{P}_{n+1}(F(u_n+h_{n+1})-F(u_n)-\mathrm{D}F(u_n)[h_{n+1}]),\nonumber\\	r_{n}&:=\epsilon^{2p}\alpha\mathcal{P}_{n+1}\mathcal{P}^{\perp}_{n}F(u_n)-\epsilon^{2-p}\mathcal{P}_{n+1}\mathcal{P}^{\perp}_{n}g.\nonumber
\end{align}
The following corlllary addresses the invertibility of $\mathcal{L}_{n+1}(\epsilon,\omega,u_n)$.
\begin{coro}%\label{le:truncated}
Given that the assumption \eqref{prior} holds at the $n$-th step, the truncated operator $\mathcal{L}_{n+1}(\epsilon,\omega,u_n)$ is invertible with
\begin{align}\label{L:inverse1}
\|\mathcal{L}^{-1}_{n+1}(\epsilon,\omega,u_n)h_{n+1}\|_{\rho_{n+1},s}
\leq{K}_1\delta^{-1}\|h_{n+1}\|_{\rho_{n+1},s}
\end{align}
for $h_{n+1}\in H_{n+1}$.
\end{coro}
\begin{proof}
Formula \eqref{L:inverse1} can be obtained from Proposition \ref{le:L}.
\end{proof}

According to  \eqref{L:inverse1},  we reformulate solving equation \eqref{Pn+1} as  the fixed point problem of
\begin{align*}
h_{n+1}=\mathcal{U}_{n+1}(h_{n+1}),
\end{align*}
where
\begin{align*}
\mathcal{U}_{n+1}:H_{n+1} &\rightarrow H_{n+1},\\
h_{n+1}&\mapsto -\mathcal{L}^{-1}_{n+1}(\epsilon,\omega,u_n)(R_{n}(h_{n+1})+r_{n}).
\end{align*}

The following lemma demonstrates that the mapping $\mathcal{U}_{n+1}$ is indeed a contraction.
\begin{lemm}
Let $\rho_0>0,s\geq s_0$, and $p\in\left(\frac{1}{2},1\right)$. Then for fixed $\gamma>\nu-1$, there exists  $\kappa_4\in(0,\kappa_3)$ small enough such that for any $\delta>0$  with $\delta^{\min\{2p-1,1-p\}}\gamma\leq\kappa_4$, when $\epsilon$ belongs to $\Lambda_{\delta}$,  the mapping $\mathcal{U}_{n+1}$ is a contraction in
\begin{align}
\mathcal{B}(0,\eta_{n+1}):=\left\{h_{n+1}\in H_{n+1}:\|h_{n+1}\|_{\rho_{n+1},s}\leq \eta_{n+1}\right\},\nonumber
\end{align}
where
\begin{align*}
\eta_{n+1}=\delta^{\min\{2p-1,1-p\}}\Theta N^{12}_{0}e^{-\chi^{n+1}}.
\end{align*}
\end{lemm}
\begin{proof}
Combining \eqref{initial} with the assumption \eqref{iteration} holds up to the $n$-th step, it can be derived that $\|u_n\|_{\rho_{n},s}\leq1$. Then the definition of the norm $\|\cdot\|_{\rho,s}$ implies that
\begin{align}
\|u_n\|_{\rho_{n+1},s}\leq1.\nonumber
\end{align}
Therefore, thanks to  \eqref{E:nonlinearity3} and the expression of $R_{n}(h_{n+1})$, we can deduce that for $\|h_{n+1}\|_{\rho_{n+1},s}\leq1$,
\begin{align*}			\|R_{n}(h_{n+1})\|_{\rho_{n+1},s}&\leq\delta^{2p}\alpha\|F(u_n+h_{n+1})-F(u_n)-\mathrm{D}F(u_n)[h_{n+1}]\|_{\rho_{n+1},s}\nonumber\\			&\leq\delta^{2p}\tilde{C}(s)(\|h_{n+1}\|^2_{\rho_{n+1},s+4}\|u_n\|_{\rho_{n+1},s+4}+\|h_{n+1}\|^3_{\rho_{n+1},s+4})\nonumber\\		&\leq\delta^{2p}\tilde{C}(s)(N^8_{n+1}\|h_{n+1}\|^2_{\rho_{n+1},s}N^4_n\|u_n\|_{\rho_{n+1},s}+N^{12}_{n+1}\|h_{n+1}\|^3_{\rho_{n+1},s})\nonumber\\
&\leq2\delta^{2p}\tilde{C}(s)N^{12}_{n+1}\|h_{n+1}\|^2_{\rho_{n+1},s}\nonumber.
\end{align*}
Moreover, using  \eqref{E:nonlinearity1} and \eqref{sequence}  yields that
\begin{align*}
\|\epsilon^{2p}\alpha\mathcal{P}_{n+1}\mathcal{P}^{\perp}_{n}F(u_n)\|_{\rho_{n+1},s}&\leq\delta^{2p}\alpha \|\mathcal{P}^{\perp}_{n}F(u_n)\|_{\rho_{n+1},s}\\
&\leq\delta^{2p} \tilde{C}(s)e^{-\frac{\theta}{(n+1)^2}N_n}\|u_n\|^3_{\rho_n,s+4}\\
&\leq\delta^{2p} \tilde{C}(s)N^{12}_ne^{-\frac{\theta}{(n+1)^2}N_n}
\end{align*}
and
\begin{align*}
\|\epsilon^{2-p}\mathcal{P}_{n+1}\mathcal{P}^{\perp}_{n}g\|_{\rho_{n+1},s}
&\leq\delta^{2-p}\|\mathcal{P}^{\perp}_{n}g\|_{\rho_{n+1},s}\\
&\leq\delta^{2-p} e^{-\frac{\theta}{(n+1)^2}N_n}\|g\|_{\rho_0,s}.
\end{align*}
The two inequalities lead to
\begin{align}
\|r_n\|_{\rho_{n+1},s}		
\leq\delta^{\min\{2p,2-p\}}(\delta^{2p-{\min\{2p,2-p\}}}\tilde{C}(s)+\delta^{2-p-{\min\{2p,2-p\}}}\|g\|_{\rho_0,s})N^{12}_{n}e^{-\frac{\theta}{(n+1)^2}N_n}.\nonumber
\end{align}
Following the above calculation and applying \eqref{L:inverse1}, we have
\begin{align*}
\|\mathcal{U}_{n+1}(h_{n+1})\|_{\rho_{n+1},s}
&\leq{K}_1\delta^{-1}\|R_{n}(h_{n+1})\|_{\rho_{n+1},s}+
{K}_1\delta^{-1}\|r_{n}\|_{\rho_{n+1},s}\\
&\leq2\delta^{2p-1}{K}_1\tilde{C}(s)N^{12}_{n+1}\|h_{n+1}\|^2_{\rho_{n+1},s}+\delta^{\min\{2p-1,1-p\}} {K}_1(\delta^{2p-{\min\{2p,2-p\}}}\tilde{C}(s)\\
&\quad+\delta^{2-p-{\min\{2p,2-p\}}}\|g\|_{\rho_0,s})N^{12}_{n}e^{-\frac{\theta}{(n+1)^2}N_n}.
\end{align*}
If $\frac{1}{2}<p<\frac{2}{3}$, then $\min\{2p,2-p\}=2p$, which leads to
\begin{align*}
2p-\min\{2p,2-p\}=0,\quad 2-p-\min\{2p,2-p\}>0.
\end{align*}
When $\frac{2}{3}\leq p<1$, one has $\min\{2p,2-p\}=2-p$. It is straightforward that
\begin{align*}
2p-\min\{2p,2-p\}\geq0,\quad 2-p-\min\{2p,2-p\}=0.
\end{align*} 
Since
\begin{align}
{K}_1(\delta^{2p-{\min\{2p,2-p\}}}\tilde{C}(s)+\delta^{2-p-{\min\{2p,2-p\}}}\|g\|_{\rho_0,s})
<\frac{1}{2}\bar\Theta,\nonumber
\end{align}
it follows from \eqref{Nn} and \eqref{sequence} that there exists some $N_0>0$ large enough,
\begin{align}
&2\delta^{2p-1}{K}_1\tilde{C}(s)N^{12}_{n+1}{\eta^2_{n+1}}\leq\frac12\eta_{n+1},\label{eta_n+1}\\
&\delta^{\min\{2p-1,1-p\}} {K}_1(\delta^{2p-{\min\{2p,2-p\}}}\tilde{C}(s)+\delta^{2-p-{\min\{2p,2-p\}}}\|g\|_{\rho_0,s})N^{12}_{n}e^{-\frac{\theta}{(n+1)^2}N_n}\leq\frac12\eta_{n+1}.\nonumber
\end{align}
This shows that $\|\mathcal{U}_{n+1}(h_{n+1})\|_{\rho_{n+1},s}\leq \eta_{n+1}$.
		
Additionally, by means of the expression of $R_{n}(h_{n+1})$, we obtain
\begin{align*}
&R_{n}(h_{n+1}+\xi \mathrm{w})-R_{n}(h_{n+1})\\
&=\epsilon^{2p}\alpha\mathcal{P}_{n+1}(F(u_n+h_{n+1}+\xi \mathrm{w})-F(u_n+h_{n+1})-\mathrm{D}F(u_n)[\xi \mathrm{w}]).
\end{align*}		
By taking the derivative  of  $\mathcal{U}_{n+1}$ with respect to $h$, one has
\begin{align*}
\mathrm{D}\mathcal{U}_{n+1}(h_{n+1})[\mathrm{w}]&=\frac{\mathrm{d}}{\mathrm{d}\xi}
\mathcal{U}_{n+1}(h_{n+1}+\xi \mathrm{w})\Big|_{\xi=0}\\
&=\frac{\mathrm{d}}{\mathrm{d}\xi}
(-\mathcal{L}^{-1}_{n+1}(\epsilon,\omega,u_n)(R_{n}(h_{n+1}+\xi \mathrm{w})+r_{n}))\Big|_{\xi=0}\\
&=-\lim\limits_{\xi \rightarrow0}\frac{1}{\xi}\mathcal{L}^{-1}_{n+1}(\epsilon,\omega,u_n)(R_{n}(h_{n+1}+\xi \mathrm{w})-R_{n}(h_{n+1}))\\	&=\epsilon^{2p}\alpha\mathcal{L}^{-1}_{n+1}(\epsilon,\omega,u_n)\mathcal{P}_{n+1}(-\mathrm{D}F(u_n+h_{n+1})[\mathrm{w}]+\mathrm{D}F(u_n)[\mathrm{w}]).
\end{align*}
Because of \eqref{E:nonlinearity2}, \eqref{L:inverse1}, and \eqref{eta_n+1}, we get
\begin{align*}
\|\mathrm{D}\mathcal{U}_{n+1}(h_{n+1})[\mathrm{w}]\|_{\rho_{n+1},s}	%&\leq\delta^\frac32\|\mathcal{L}^{-1}_{n+1}(\epsilon,\omega,v_n)\mathcal{P}_{n+1}
%(-\mathrm{D}F(v_n+h_{n+1})[\mathrm{w}]+\mathrm{D}F(v_n)[\mathrm{w}])\|_{\rho_{n+1},s}\\
&\leq\delta^{2p-1}{K}_1\tilde{C}(s)N^{12}_{n+1}\|h_{n+1}\|_{\rho_{n+1},s}\|\mathrm{w}\|_{\rho_{n+1},s}\\
&\leq\delta^{2p-1}{K}_1\tilde{C}(s)N^{12}_{n+1}\eta_{n+1}\|\mathrm{w}\|_{\rho_{n+1},s}\\
&\leq\frac12\|\mathrm{w}\|_{\rho_{n+1},s}.
\end{align*}
Hence, $\mathcal{U}_{n+1}$ is a contraction in $\mathcal{B}(0,\eta_{n+1})$.
\end{proof}

The remainder of the proof is to verify formula \eqref{prior} at the $(n+1)$-th step. Let us define
\begin{align*}
C(\tau):=\big(\frac{\tau+6}{e}\big)^{\tau+6}\frac{1}{\theta^{\tau+6}}\big(1+\sum^{+\infty}_{l=1}(l+1)^{2(\tau+6)}e^{-\chi^{l}}\big).
\end{align*}
%Using the fact $\rho_l-\rho_{n+1}\geq \rho_l-\rho_{l+1}, l=1,2,\cdots,n$,
According to \eqref{sequence}, \eqref{initial}, \eqref{iteration}, and \eqref{1}, there is $\tau\in(0,\tau_4)$ such that $\delta^{\min\{2p-1,1-p\}}\gamma\leq\kappa\leq\kappa_4$ small enough,  the following holds:
\begin{align*}
\|u_{n+1}\|_{\rho_{n+2},s+\tau+6}%&=	\|u_{n+1}\|_{\lambda^{n+1}\rho_{0},s+4}\\
&\leq\|u_0\|_{_{\rho_{n+2},s+\tau+6}}+\sum^{n+1}_{l=1}\|h_l\|_{\rho_{n+2},s+\tau+6}\\
&\leq\big(\frac{\tau+6}{e}\big)^{\tau+6}(\frac{\|u_0\|_{\rho_0,s}}{(\rho_0-\rho_1)^{\tau+6}}+\sum^{n+1}_{l=1}\frac{\|h_l\|_{\rho_{l},s}}{(\rho_{l}-\rho_{l+1})^{\tau+6}})\\
&\leq C(\tau)\delta^{\min\{2p-1,1-p\}}\bar\Theta N^{12}_0\\
&\leq1.
\end{align*}
The proof of Lemma \ref{le:iterative} is complete.
\end{proof}

\subsection{Proof of Theorem 2.2}%\label{sec:3.3}

From the inductive lemma \ref{le:iterative}, we can construct the quasi-periodic solution with the  frequency $\omega$ for equation \eqref{E:Func}, namely, we complete the proof of Theorem \ref{main-result}..
%The purpose of this subsection is to complete the proof of Theorem \ref{main-result}.
%Theorem \ref{main-result} follows from Lemma \ref{le:iterative}.
\begin{proof}
In terms of  Lemma \ref{le:iterative}, we define
\begin{align*}
u:=u_0+\sum^{\infty}_{n=1}h_n.
\end{align*}
Therefore, equation \eqref{kirchhoff1} possesses a solution $u\in H^{{\rho_0}/{2},s}_0$ with zero average, which is quasi-periodic in time.
	
The proof is complete.
\end{proof}

\appendix
\section{Auxiliary results}
  The purpose of  the Appendix is devoted to introducing some technical lemmas. The following lemma corresponds to the properties of the composition operator.
\begin{lemm}\label{le:composition}
Let $\rho>0$ and $s>{\nu}/{2}$. For $u\in H^{\rho,s+4}_0$, define the mapping
\begin{align}
 F: \quad H^{\rho,s+4}_0\rightarrow {H}^{\rho,s}_0,\quad u\mapsto(\omega\cdot\nabla_{\varphi})u(\lambda\int_{\mathbb{T}^d}|\Delta u|^2 \mathrm{d}x+\int_{\mathbb{T}^d}|(\omega\cdot\nabla_{\varphi}) u|^2 \mathrm{d}x).\nonumber
\end{align}
Then $F$ is analytic from $H^{\rho,s+4}_0$ to $H^{\rho,s}_0$ with
\begin{align}
\mathrm{D}F(u)[h]=&\Big(2\lambda\int_{\mathbb{T}^d}\Delta u\cdot\Delta h \mathrm{d}x+2\int_{\mathbb{T}^d}(\omega\cdot\nabla_{\varphi}) u\cdot(\omega\cdot\nabla_{\varphi}) h \mathrm{d}x\Big)(\omega\cdot\nabla_{\varphi})u\nonumber\\
&+\Big(\lambda\int_{\mathbb{T}^d}|\Delta u|^2 \mathrm{d}x+\int_{\mathbb{T}^d}|(\omega\cdot\nabla_{\varphi}) u|^2 \mathrm{d}x\Big)(\omega\cdot\nabla_{\varphi})h.\nonumber
\end{align}

Furthermore, there exists some constant $C(s)>0$ such that for all $u,h\in H^{\rho,s+4}_0$,
\begin{align}\label{E:nonlinearity1}
\|F(u)\|_{\rho,s}&\leq C(s)\|u\|^3_{\rho,s+4},\\
\|\mathrm{D}F(u)[h]\|_{\rho,s}&\leq C(s)\|u\|^2_{\rho,s+4}\|h\|_{\rho,s+4},\label{E:nonlinearity2}\\
\|F(u+h)-F(u)-\mathrm{D}F(u)[h]\|_{\rho,s}&\leq C(s)(\|u\|_{\rho,s+4}\|h\|^2_{\rho,s+4}+\|h\|^3_{\rho,s+4}).\label{E:nonlinearity3}
\end{align}
\end{lemm}
\begin{proof}
Clearly, $F$ is an analytic mapping from $H^{\rho,s+4}_0$ to  $H^{\rho,s}_0$.  Moreover, notice that
\begin{align*}
\mathrm{D}F(u)[h]&=\frac{\mathrm{d}}{\mathrm{d}\xi}F(u+\xi h)\Big|_{\xi=0}\\
&=\frac{\mathrm{d}}{\mathrm{d}\xi}(\omega\cdot\nabla_{\varphi})(u+\xi h)(\lambda\int_{\mathbb{T}^d}|\Delta (u+\xi h)|^2 \mathrm{d}x+\int_{\mathbb{T}^d}|(\omega\cdot\nabla_{\varphi}) (u+\xi h)|^2 \mathrm{d}x)\Big|_{\xi=0}\\
&=\lim\limits_{\xi \rightarrow0}\frac{1}{\xi}((\omega\cdot\nabla_{\varphi})(u+\xi h)(\lambda\int_{\mathbb{T}^d}|\Delta (u+\xi h)|^2 \mathrm{d}x+\int_{\mathbb{T}^d}|(\omega\cdot\nabla_{\varphi}) (u+\xi h)|^2 \mathrm{d}x)\\
&\quad\quad-(\omega\cdot\nabla_{\varphi})u(\lambda\int_{\mathbb{T}^d}|\Delta u|^2 \mathrm{d}x+\int_{\mathbb{T}^d}|(\omega\cdot\nabla_{\varphi}) u|^2 \mathrm{d}x))\\
&=\Big(2\lambda\int_{\mathbb{T}^d}\Delta u \cdot\Delta h\mathrm{d}x+2\int_{\mathbb{T}^d}(\omega\cdot\nabla_{\varphi}) u\cdot(\omega\cdot\nabla_{\varphi}) h \mathrm{d}x\Big)(\omega\cdot\nabla_{\varphi})u \\
&\quad+\Big(\lambda\int_{\mathbb{T}^d}|\Delta u|^2 \mathrm{d}x+\int_{\mathbb{T}^d}|(\omega\cdot\nabla_{\varphi}) u|^2 \mathrm{d}x)(\omega\cdot\nabla_{\varphi}\Big)h.
\end{align*}

We are now in a position to check \eqref{E:nonlinearity1}. Let $a(\varphi):=A(\varphi)+B(\varphi),$ where
\begin{align*}
A(\varphi)=\lambda\int_{\mathbb{T}^d}|\Delta u|^2 \mathrm{d}x\quad B(\varphi)=\int_{\mathbb{T}^d}|(\omega\cdot\nabla_{\varphi}) u|^2 \mathrm{d}x.
\end{align*}
In Fourier bases $e^{\mathrm{i}k\cdot\varphi}$ and $e^{\mathrm{i}j\cdot x}$, we can derive
\begin{align*}
A(\varphi)(\omega\cdot\nabla_{\varphi})u(\varphi,x)&=\sum_{k\in\mathbb{Z}^\nu}\mathrm{i}(\omega\cdot k)A(\varphi){u}_k(x)e^{\mathrm{i}k\cdot\varphi}\\
&=\sum_{k,l\in\mathbb{Z}^\nu,j\in\mathbb{Z}^d}\mathrm{i}(\omega\cdot k){A}_{l-k}{u}_{k,j}e^{\mathrm{i}l\cdot\varphi}e^{\mathrm{i}j\cdot x},\\
B(\varphi)(\omega\cdot\nabla_{\varphi})u(\varphi,x)&=\sum_{k\in\mathbb{Z}^\nu}\mathrm{i}(\omega\cdot k)B(\varphi){u}_k(x)e^{\mathrm{i}k\cdot\varphi}\\
&=\sum_{k,l\in\mathbb{Z}^\nu,j\in\mathbb{Z}^d}\mathrm{i}(\omega\cdot k){B}_{l-k}{u}_{k,j}e^{\mathrm{i}l\cdot\varphi}e^{\mathrm{i}j\cdot x},
\end{align*}
which lead to
\begin{align}	\|A(\omega\cdot\nabla_{\varphi})u\|_{\rho,s}^2=\sum_{l\in\mathbb{Z}^\nu,j\in\mathbb{Z}^d}\Big|\sum_{k\in\mathbb{Z}^{\nu}}\mathrm{i}(\omega\cdot k){A}_{l-k}{u}_{k,j}\Big|^2\langle l,j\rangle^{2s}e^{2\rho(|l|+|j|)},\nonumber\\
\|B(\omega\cdot\nabla_{\varphi})u\|_{\rho,s}^2=\sum_{l\in\mathbb{Z}^\nu,j\in\mathbb{Z}^d}\Big|\sum_{k\in\mathbb{Z}^{\nu}}\mathrm{i}(\omega\cdot k){B}_{l-k}{u}_{k,j}\Big|^2\langle l,j\rangle^{2s}e^{2\rho(|l|+|j|)}.\nonumber
\end{align}
Since
\begin{align}
A(\varphi)=\lambda(2\pi)^d\sum_{k,l\in\mathbb{Z}^{\nu},j\in\mathbb{Z}^d}
|j|^4u_{k-l,j}u_{l,-j}e^{\mathrm{i}k\cdot\varphi},\nonumber
\end{align}
one has
\begin{align}	
\|A\|^2_{\varphi,\rho,s}=\lambda^{2}(2\pi)^{2d}\sum_{k\in\mathbb{Z}^{\nu}}\Big|\sum_{l\in\mathbb{Z}^{\nu},j\in\mathbb{Z}^d}
|j|^4u_{k-l,j}{u}_{l,-j}\Big|^2\langle k\rangle^{2s}e^{2\rho|k|}.\nonumber
\end{align}
Observe that for $s>{\nu}/{2}$,
\begin{align}
\sum_{l\in\mathbb{Z}^{\nu}}\frac{\langle k\rangle^{2s}}{\langle k-l\rangle^{2s}\langle l\rangle^{2s}}\leq2^{2s-1}\sum_{l\in\mathbb{Z}^{\nu}}\frac{1}{\langle l\rangle^{2s}}+2^{2s-1}\sum_{l\in\mathbb{Z}^{\nu}}\frac{1}{\langle k-l\rangle^{2s}}=:C^2_1(s).\nonumber
\end{align}
It follows from  the Cauchy inequality that
\begin{align}	
\sum_{j\in\mathbb{Z}^d}|j|^4|u_{k-l,j}||{u}_{l,-j}|\leq\Big(\sum_{j\in\mathbb{Z}^d}
|j|^8|{u}_{k-l,j}|^2\Big)^{\frac12}\Big(\sum_{j\in\mathbb{Z}^d}|{u}_{l,-j}|^2\Big)^{\frac12}.\nonumber
\end{align}
This shows that
\begin{align*}
\Big|\sum_{l\in\mathbb{Z}^{\nu},j\in\mathbb{Z}^d}|j|^4{u}_{k-l,j}{u}_{l,-j}\Big|^2\langle k\rangle^{2s}
&\leq\Big|\sum_{l\in\mathbb{Z}^{\nu}}\Big(\sum_{j\in\mathbb{Z}^d}|j|^8|{u}_{k-l,j}|^2\Big)^{\frac12}
\Big(\sum_{j\in\mathbb{Z}^d}|u_{l,-j}|^2\Big)^{\frac12}\langle k\rangle^{s}\Big|^2\\
&\leq C^2_1(s)\sum_{l\in\mathbb{Z}^{\nu}}\Big(\sum_{j\in\mathbb{Z}^d}|j|^8|{u}_{k-l,j}|^2|\langle k-l\rangle^{2s}\Big)\Big(\sum_{j\in\mathbb{Z}^d}|{u}_{l,-j}|^2\langle l\rangle^{2s}\Big),
\end{align*}
which then gives
\begin{align}\label{a-varphi}
\|A\|^2_{\varphi,\rho,s}
&\leq C^2_2(s)\sum_{k,l\in\mathbb{Z}^{\nu}}\Big(\sum_{j\in\mathbb{Z}^d}|j|^8|{u}_{k-l,j}|^2\langle k-l\rangle^{2s}\Big)
\Big(\sum_{j\in\mathbb{Z}^d}|{u}_{l,-j}|^2\langle l\rangle^{2s}\Big)e^{2\rho|k|}\nonumber\\
&\leq C^2_2(s)\sum_{l\in\mathbb{Z}^{\nu}}\Big(\sum_{j\in\mathbb{Z}^d}|u_{l,-j}|^2\langle l,j\rangle^{2s}e^{2\rho(|l|+|j|)}\Big)\Big(\sum_{k\in\mathbb{Z}^{\nu},j\in\mathbb{Z}^d}|u_{k-l,j}|^2\langle k-l,j\rangle^{2s+8}e^{2\rho(|k-l|+|j|)}\Big)\nonumber\\
&\leq C^2_2(s)\|u\|^2_{\rho,s+4}\|u\|^2_{\rho,s}.
\end{align}
Furthermore, by applying the Cauchy inequality, we can obtain that  for $s>{\nu}/{2}$,
\begin{align*}
\Big|\sum_{k\in\mathbb{Z}^{\nu}}\mathrm{i}(\omega\cdot k)A_{l-k}u_{k,j}\Big|^2\langle l,j\rangle^{2s}&=\Big|\sum_{k\in\mathbb{Z}^{\nu}}\mathrm{i}(\omega\cdot k)A_{l-k}u_{k,j}\langle l,j\rangle^{s}\Big|^2\\
&\leq C_1^2(s)\sum_{k\in\mathbb{Z}^\nu}|A_{l-k}|^2\langle l-k\rangle^{2s}|u_{k,j}|^2\langle k,j\rangle^{2(s+1)}.
\end{align*}
Consequently, we have
\begin{align}\label{omegaA}	
\|A(\omega\cdot\nabla_{\varphi})u\|_{\rho,s}^2
&\leq C^2_1(s)\sum_{l,k\in\mathbb{Z}^\nu,j\in\mathbb{Z}^d}|A_{l-k}|^2\langle l-k\rangle^{2s}|u_{k,j}|^2\langle
k,j\rangle^{2(s+1)}e^{2\rho(|l|+|j|)}\nonumber\\
&\leq C^2_1(s)\sum_{k\in\mathbb{Z}^\nu,j\in\mathbb{Z}^d}|u_{k,j}|^2\langle k,j\rangle^{2(s+1)}e^{2\rho(|k|+|j|)}(\sum_{l\in\mathbb{Z}^{\nu}}|A_{l-k}|^2\langle l-k\rangle^{2s}e^{2\rho|l-k|})\nonumber\\
&=C^2_1(s)\|A\|^{2}_{\varphi,\rho,s}\sum_{k\in\mathbb{Z}^\nu,j\in\mathbb{Z}^d}|u_{k,j}|^2\langle k,j\rangle^{2(s+1)}e^{2\rho(|k|+|j|)}\nonumber\\
&\leq C^2_3(s)\|u\|^2_{\rho,s+1}\|u\|^2_{\rho,s+4}\|u\|^2_{\rho,s}.
\end{align}
%Additionally,
%\begin{align}
%B(\varphi)=(2\pi)^d\sum_{k\in\mathbb{Z}^{\nu},j\in\mathbb{Z}^d}
%|\omega\cdot k|^2u_{k,j}u_{-k,-j},\nonumber
%\end{align}
%we have
%\begin{align}	
%\|B\|^2_{\varphi,\rho,s}=(2\pi)^{2d}\sum_{k\in\mathbb{Z}^{\nu}}\Big|\sum_{j\in\mathbb{Z}^d}
%|\omega\cdot k|^2u_{k,j}{u}_{-k,-j}\Big|^2.\nonumber
%\end{align}
%It follows from  the Cauchy inequality that
%\begin{align}	
%\sum_{j\in\mathbb{Z}^d}|\omega\cdot k|^2|u_{k,j}||{u}_{-k,-j}|\leq\Big(\sum_{j\in\mathbb{Z}^d}
%|\omega\cdot k|^4|{u}_{k,j}|^2\Big)^{\frac12}\Big(\sum_{j\in\mathbb{Z}^d}|{u}_{-k,-j}|^2\Big)^{\frac12}.\nonumber
%\end{align}
Proceeding a similar procedure as above, we have
\begin{align}\label{B-varphi}	
\|B\|^2_{\varphi,\rho,s}
%&\leq (2\pi)^{2d}\sum_{k\in\mathbb{Z}^{\nu}}\Big(\sum_{j\in\mathbb{Z}^d}
%|\omega\cdot k|^4|{u}_{k,j}|^2\Big)\Big(\sum_{j\in\mathbb{Z}^d}|{u}_{-k,-j}|^2\Big).\nonumber\\
%&\leq(2\pi)^{2d}\sum_{k\in\mathbb{Z}^{\nu}}\Big(\sum_{j\in\mathbb{Z}^d}
%\langle k,j\rangle^4|{u}_{k,j}|^2\Big)\Big(\sum_{j\in\mathbb{Z}^d}|{u}_{-k,-j}|^2\langle k,j\rangle^{2s}\langle k,j\rangle^{-2s}\Big)\nonumber\\
%&\leq(2\pi)^{2d}\sum_{k\in\mathbb{Z}^{\nu}}\Big(\sum_{j\in\mathbb{Z}^d}
%\langle k,j\rangle^{2(s+2)}|{u}_{k,j}|^2e^{2\rho(|k|+|j|)}e^{-2\rho(|k|+|j|)}\Big)\nonumber\\
%&\quad\times\Big(\sum_{j\in\mathbb{Z}^d}|{u}_{k,j}|^2\langle k,j\rangle^{2s}\langle k,j\rangle^{-2s}e^{2\rho(|k|+|j|)}e^{-2\rho(|k|+|j|)}\Big)\nonumber\\
&\leq C^2_4(s)\|u\|^2_{\rho,s+2}\|u\|^2_{\rho,s},\\
\|B(\omega\cdot\nabla_{\varphi})u\|_{\rho,s}^2&\leq C^2_5(s)\|u\|^2_{\rho,s+1}\|u\|^2_{\rho,s+2}\|u\|^2_{\rho,s}.\label{omegaB}
\end{align}
%Furthermore, by applying the Cauchy inequality, we can obtain that  for $s>{\nu}/{2}$,
%\begin{align*}
%\Big|\sum_{k\in\mathbb{Z}^{\nu}}\mathrm{i}(\omega\cdot k)B_{l-k}u_{k,j}\Big|^2\langle l,j\rangle^{2s}&=\Big|\sum_{k\in\mathbb{Z}^{\nu}}\mathrm{i}(\omega\cdot k)B_{l-k}u_{k,j}\langle l,j\rangle^{s}\Big|^2\\
%&\leq C_1^2(s)\sum_{k\in\mathbb{Z}^\nu}|B_{l-k}|^2\langle l-k\rangle^{2s}|u_{k,j}|^2\langle k,j\rangle^{2(s+1)}.
%\end{align*}
%Consequently, we have
%\begin{align}\label{omegaB}		
%\|B(\omega\cdot\nabla_{\varphi})u\|_{\rho,s}^2
%&\leq C^2_1(s)\sum_{l,k\in\mathbb{Z}^\nu,j\in\mathbb{Z}^d}|B_{l-k}|^2\langle l-k\rangle^{2s}|u_{k,j}|^2\langle
%k,j\rangle^{2(s+1)}e^{2\rho(|l|+|j|)}\nonumber\\
%&\leq C^2_1(s)\sum_{k\in\mathbb{Z}^\nu,j\in\mathbb{Z}^d}|u_{k,j}|^2\langle k,j\rangle^{2(s+1)}e^{2\rho(|k|+|j|)}\nonumber\\
%&\quad\times(\sum_{l\in\mathbb{Z}^{\nu}}|B_{l-k}|^2\langle l-k\rangle^{2s}e^{2\rho|l-k|})\nonumber\\
%&=C^2_1(s)\|B\|^{2}_{\varphi,\rho,s}\sum_{k\in\mathbb{Z}^\nu,j\in\mathbb{Z}^d}|u_{k,j}|^2\langle k,j\rangle^{2(s+1)}e^{2\rho(|k|+|j|)}\nonumber\\
%&\leq C^2_5(s)\|u\|^2_{\rho,s+1}\|u\|^2_{\rho,s+2}\|u\|^2_{\rho,s}.
%\end{align}
Combining \eqref{omegaA} and \eqref{omegaB}, we conclude that \eqref{E:nonlinearity1} holds.

A similar argument as above yields that
\begin{align*}
&\|(\omega\cdot\nabla_{\varphi})u(\lambda\int_{\mathbb{T}^d}\Delta u\cdot\Delta
h\mathrm{d}x+\int_{\mathbb{T}^d}(\omega\cdot\nabla_{\varphi}) u\cdot(\omega\cdot\nabla_{\varphi})h\mathrm{d}x)\|_{\rho,s}\leq C(s)\|u\|^2_{\rho,s+4}\|h\|_{\rho,s},\\
&\|(\omega\cdot\nabla_{\varphi})h(\lambda\int_{\mathbb{T}^d}|\Delta u|^2\mathrm{d}x+\int_{\mathbb{T}^d}|(\omega\cdot\nabla_{\varphi}) u|^2\mathrm{d}x)\|_{\rho,s}\leq C(s)\|u\|_{\rho,s+4}\|u\|_{\rho,s}\|h\|_{\rho,s+1}.
\end{align*}
Hence we obtain \eqref{E:nonlinearity2}.
	
In addition, note that
\begin{align*}
F(u+h)-F(u)-\mathrm{D}F(u)[h]
&=(\omega\cdot\nabla_{\varphi})u(\lambda\int_{\mathbb{T}^d}|\Delta h|^2\mathrm{d}x+\int_{\mathbb{T}^d}|(\omega\cdot\nabla_{\varphi})h|^2\mathrm{d}x)\\
&\quad+(\omega\cdot\nabla_{\varphi})h(\lambda\int_{\mathbb{T}^d}|\Delta h|^2\mathrm{d}x+\int_{\mathbb{T}^d}|(\omega\cdot\nabla_{\varphi})h|^2\mathrm{d}x\\
&\quad+2\lambda\int_{\mathbb{T}^d}\Delta u\cdot\Delta h\mathrm{d}x+2\int_{\mathbb{T}^d}(\omega\cdot\nabla_{\varphi})u
\cdot(\omega\cdot\nabla_{\varphi})h\mathrm{d}x).\nonumber
\end{align*}
Similarly, we can deduce \eqref{E:nonlinearity3}.
	
The proof is complete.
\end{proof}

Let $u$ be a real analytic function on $\mathbb{T}^{m}_\rho$ with $\int_{\mathbb{T}^m}u(z)\mathrm{d}z=0$. The max norm of $u$ is defined by
\begin{align}
|u|_{z,\rho,s}:=\sum_{l\in\mathbb{N},|l|\leq s}\max_{z\in\mathbb{T}^m_\rho}|\mathrm{D}^l u(z)|.\nonumber
\end{align}
For $l\in\mathbb{N}$, the max norm of $\mathrm{D}^{l}u$ on $\mathbb{T}^{m}_\rho$ is
\begin{align}\label{N:max}	|\mathrm{D}^{l}u|_{z,\rho}=\sum_{\alpha\in\mathbb{Z}^m,|\alpha|=l}|\mathrm{D}^{\alpha}u|_{z,\rho}=\sum_{\alpha\in\mathbb{Z}^m,|\alpha|=l}\max_{z\in\mathbb{T}^m_{\rho}}|\mathrm{D}^{\alpha}u(z)|.
\end{align}
Moreover, we fix $s_1>{m}/{2}$ in the Appendix.

In the following lemma, we are devoted to discussing the relationship between the Sobolev norm $\|\cdot\|_{z,\rho,s}$ and  the max norm $|\cdot|_{z,\rho,s}$ (cf. \cite{Baldi2014KAM,sun2018quasi}).

\begin{lemm}%\label{le:max}
Let $u$ be a real analytic function on $\mathbb{T}^{m}_\rho$ with $\int_{\mathbb{T}^m}u(z)\mathrm{d}z=0$. For $\rho>0,s\geq 0$,
 there exists a constant ${C}>0$ such that for $u\in H^{\rho,s+s_1}_{z,0}$, the following inequalities hold:
\begin{align}\label{2}
|u|_{z,\rho,s}\leq {C}\|u\|_{z,\rho,s+s_1},\quad \|u\|_{z,\rho,s}\leq C|u|_{z,\rho,s+s_1}.
\end{align}
\end{lemm}

\begin{proof}
Since  $\int_{\mathbb{T}^m}u(z)\mathrm{d}z=0$, which implies $u_0=0$, we can obtain
\begin{align*}
|u|_{z,\rho,s}&\leq C_1\sum_{k\in\mathbb{Z}^m}|u_k|\langle k\rangle^{s}e^{\rho|k|}\\
&\leq C_1(\sum_{k\in\mathbb{Z}^m}|u_k|^2\langle k\rangle^{2(s+s_1)}e^{2\rho|k|})^{\frac{1}{2}}(\sum_{k\in\mathbb{Z}^m}\frac{1}{\langle k\rangle^{2s_1}})^{\frac12}\\
&\leq C\|u\|_{z,\rho,s+s_1}.
\end{align*}
Therefore, the first inequality in \eqref{2} holds.

It remains to verify the second inequality in \eqref{2}. Since $u$ is analytic on $\mathbb{T}^m_{\rho}$, the Fourier coefficients $u_k$ have the following decay property:
\begin{align}
|u_k|\leq|u|_{z,\rho}e^{-\rho|k|}.\nonumber
\end{align}
Moreover, $\mathrm{D}^{\alpha}u$ is also analytic on $\mathbb{T}^m_{\rho}$ with
\begin{align}
(\mathrm{D}^\alpha u)_{k}=u_k(\mathrm{i}k)^{\alpha},\quad (\mathrm{i}k)^{\alpha}=(\mathrm{i}k_1)^{\alpha_1}\cdots(\mathrm{i}k_m)^{\alpha_m}.\nonumber
\end{align}
As a consequence,
\begin{align}
|(\mathrm{D}^\alpha u)_{k}|=|u_k||(\mathrm{i}k)^{\alpha}|\leq |\mathrm{D}^\alpha u|_{z,\rho}e^{-\rho|k|}.\nonumber
\end{align}
According to \eqref{N:max}, if $|\alpha|=s+s_1$, then
\begin{align*}		|u_k||\mathrm{i}k|^{s+s_1}=&\sum_{|\alpha|=s+s_1}|u_k||(\mathrm{i}k)^{\alpha}|=\sum_{|\alpha|=s+s_1}|(\mathrm{D}^\alpha u)_{k}|\\
\leq&\sum_{|\alpha|=s+s_1}|\mathrm{D}^\alpha u|_{z,\rho}e^{-\rho|k|}=|\mathrm{D}^{s+s_1} u|_{z,\rho}e^{-\rho|k|}.
\end{align*}
Hence,
\begin{align*}
\|u\|^{2}_{z,\rho,s}&=\sum_{k\in\mathbb{Z}^m}|u_k|^2e^{2\rho|k|}\langle k\rangle^{2s}\\
&=\sum_{k\in\mathbb{Z}^m\backslash\{0\}}|u_k|^2e^{2\rho|k|}| k|^{2(s+s_1)}|k|^{-2s_1}\\
&\leq|\mathrm{D}^{s+s_1} u|^2_{z,\rho}\sum_{k\in\mathbb{Z}^m\backslash\{0\}}|k|^{-2s_1}\\
&\leq C| u|^2_{z,\rho,s+s_1},
\end{align*}
which completes the proof.
\end{proof}

\begin{lemm}%\label{le:max}
Let $u$ be a real analytic function on $\mathbb{T}^{m}_\rho$ with $\int_{\mathbb{T}^m}u(z)\mathrm{d}z=0$. For $0\leq\breve{\rho}<\rho,\sigma>0,s\geq 0$, if $u\in H^{\rho,s}_{z,0}$, then
\begin{align}\label{1}
&\|u\|_{\rho-\breve{\rho},s+\sigma}\leq\big(\frac{\sigma}{e}\big)^{\sigma}\frac{1}{\breve{\rho}^{\sigma}}\|u\|_{\rho,s}.
\end{align}
\end{lemm}
\begin{proof}
Since $\int_{\mathbb{T}^m}u(z)\mathrm{d}z=0$ and
\begin{align*}
\frac{\zeta^{\sigma}}{e^{\breve{\rho}\zeta}}\leq \big(\frac{\sigma}{e}\big)^{\sigma}\frac{1}{\breve{\rho}^{\sigma}},\quad \forall\zeta\geq0,
\end{align*}
we obtain
\begin{align*}%\label{regular}
\|u\|^2_{\rho-\breve{\rho},s+\sigma}=&\sum_{k\in\mathbb{Z}^m} |u_{k}|^2\langle k\rangle^{2(s+\sigma)} e^{2(\rho-\breve{\rho})|k|}\\
=&\sum_{k\in\mathbb{Z}^m,k\neq0} |u_{k}|^2|k|^{2s} e^{2\rho|k|}\frac{|k|^{2\sigma}}{e^{2\breve{\rho} |k|}}\\
\leq &\big(\frac{\sigma}{e}\big)^{2\sigma}\frac{1}{\breve{\rho}^{2\sigma}}\|u\|^2_{\rho,s}.
\end{align*}
This ends the proof of the lemma.
\end{proof}

%\bibliographystyle{abbrv}
%\bibliography{references}

\end{document}